\documentclass[a4paper,10pt]{article}

\usepackage{amsmath,amsthm,amssymb,latexsym,amsfonts}
\usepackage{latexcad}
\usepackage{latexsym}
\usepackage{amssymb}
\usepackage[latin1]{inputenc}
\usepackage[english]{babel}
\usepackage{layout}

\usepackage{graphicx}
\usepackage{xcolor}
\usepackage{color}

\def\ep{\varepsilon}
\def\C{\mathcal{C}}
\def\bC{\mathbb{C}}
\def\F{\mathcal{F}}
\def\R{\mathbb{R}}
\def\E{\mathbb{E}}
\def\P{\mathbb{P}}

\newtheorem{theorem}{Theorem}[section]

\newtheorem{proposition}[theorem]{Proposition}
\newtheorem{lemma}[theorem]{Lemma}
\newtheorem{definition}[theorem]{Definition}

\newcommand*\samethanks[1][\value{footnote}]{\footnotemark[#1]}

\begin{document}

\title{Weak approximation of the complex Brownian sheet \\from a L\'evy sheet
and applications to SPDEs}

\author{Xavier Bardina
\footnote{Departament de Matem\`atiques, Universitat Aut\`{o}noma de
Barcelona, 08193 Bellaterra (Barcelona), Catalonia. E-mail
addresses: bardina@mat.uab.cat, juanpablo@mat.uab.cat,
quer@mat.uab.cat. Research supported by the grant
PGC2018-097848-B-I00 of the Ministerio de Econom\'ia y
Competitividad. J.P. M\'arquez was supported by a fellowship of
CONACYT-M\'exico} \and Juan Pablo Márquez \samethanks[1] \and
Llu\'{i}s Quer-Sardanyons \samethanks[1] \footnote{Corresponding
author.}
 }

\date{\today}

\maketitle

\begin{abstract}
We consider a L\'evy process in the plane and we use it to construct a family of complex-valued random fields that
we show to converge in law, in the space of continuous functions, to a complex Brownian sheet. We apply this result to obtain weak approximations of the
random field solution to a semilinear one-dimensional stochastic heat equation driven by the space-time
white noise.
\end{abstract}

\bigskip


  \medskip


\section{Introduction}
\label{sec:intro}

Let $\{N(x,y); \, x,y\geq 0\}$ be a Poisson process in the plane and $S,T>0$. For any $\ep>0$, define the following random field:
\begin{equation}
x_\ep(s,t):=\ep
\int_0^{\frac{t}{\ep}}\int_0^{\frac{s}{\ep}}\sqrt{xy}\,
(-1)^{N(x,y)}dxdy,  \qquad (s,t)\in[0,S]\times[0,T]. \label{eq:888}
\end{equation}
 Then, in
\cite{BJ97} (see Theorem 1.1 therein) the authors proved that, as
$\ep$ tends to zero, $x_\ep$ converges in law, in the Banach space
$\C([0,S]\times[0,T])$ of continuous functions, to the Brownian
sheet on $[0,S]\times[0,T]$. It is worth mentioning that this result
was motivated by its one-dimensional counterpart, which was proved
by Stroock in \cite{Stroock} and says the following: the family of
processes
\[
 y_\ep(t):=\ep \int_0^{\frac{t}{\ep^2}} (-1)^{N(s)} ds, \; \ep>0,
\]
where $N$ denotes a standard Poisson process, converges in law, in the space of continuous functions,
to a standard Brownian motion. Note that this kind of processes had already been used by Kac in \cite{Kac}
in order to express the solution of the telegrapher's equation in terms of a Poisson process.

In the present paper, we aim to extend the above result of
\cite{BJ97} to the case where the Poisson process is replaced by a
Lévy sheet $\{L(x,y); \, x,y\geq 0\}$ (see Section \ref{sec:prel}
for the precise definition). Indeed, note that expression
$(-1)^{N(x,y)}$ can be written in terms of the complex exponential
as $e^{i\pi N(x,y)}$. Hence, when replacing $N$ by $L$, we will use
the form $e^{i\pi L(x,y)}=\cos(\pi L(x,y))+i\sin(\pi L(x,y))$ since
the expression $(-1)^{L(x,y)}$ may not be well-defined in $\R$.
On the other hand, we will replace $\pi$ by an
arbitrary angle $\theta\in(0,2\pi)$. The main result of the paper
is the following:

\begin{theorem}\label{thm:main}
Let $\{L(x,y); \, x,y\geq 0\}$ be a Lévy sheet and
$\Psi(\xi):=a(\xi)+ib(\xi)$, $\xi\in \R$, its Lévy exponent. Let
$\theta\in (0,2\pi)$ and $S,T>0$, and define, for any $\ep>0$ and
$(s,t)\in[0,S]\times[0,T]$,
\begin{equation}
X_\ep(s,t):=\ep K
\int_0^{\frac{t}{\ep}}\int_0^{\frac{s}{\ep}}\sqrt{xy}\,
\{\cos(\theta L(x,y))+i\sin(\theta L(x,y))\}dxdy, \label{eq:1}
\end{equation}
where the constant $K$ is given by
\begin{equation}\label{eq:50}
K=\frac{1}{\sqrt{2}}
\frac{a(\theta)^2+b(\theta)^2}{a(\theta) }.
\end{equation}
Assume that $a(\theta)a(2\theta)\ne 0$. Then, as $\ep$ tends to
zero, $X_\ep$ converges in law, in the space of complex-valued
continuous functions $\C([0,S]\times[0,T];\bC)$, to a complex
Brownian sheet.
\end{theorem}
We point out that the processes $X_\ep$ provide a natural extension
of the family \eqref{eq:888}. Nevertheless, the proof of Theorem
\ref{thm:main} above involves new, and indeed significantly more,
technical difficulties compared to that of \cite[Thm. 1]{BJ97}. This
is obviously due to the presence of the L\'evy sheet $L$, which is a
rather general random field that includes, for instance, the
Brownian sheet and the isotropic stable sheet (see, e.g., \cite[Sec.
2.1]{KN}).

We recall that, by definition, a complex Brownian sheet is a complex
random field whose real and imaginary parts are independent Brownian
sheets. Hence, in view of the above theorem, we observe that the
real and imaginary parts of $X_\ep$ are clearly not independent, for
any $\ep>0$, while in the limit they are. This phenomenon is not
new, for it already appeared in the study of analogous problems in
the one-parameter setting (see, e.g., \cite{B,BR,Sang}). Indeed, in
\cite{B}, a family of processes that converges in law to a complex
Brownian motion was constructed from a unique Poisson process. This
result was generalized in \cite{BR}, where the Poisson process was
replaced by processes with independent increments whose
characteristic functions satisfy some properties. L\'evy processes
are one of the examples where the latter results may be applied.
Finally, the authors of \cite{Sang} use Poisson and L\'evy processes
in order to obtain approximations in law of a complex fractional
Brownian motion.

The main strategy in order to prove the kind of weak convergence
stated in Theorem \ref{thm:main}
consists in proving that the underlying family
of laws is relatively compact in the space of continuous functions (with the usual topology).
By Prohorov's theorem, this is
equivalent to proving the tightness property of this family of laws. Next, we will check that
 every weakly convergent partial sequence
converges to the limit law that we want to obtain.

In the last part of the paper (see Section \ref{sec:heat}), we consider the following semilinear
stochastic heat equation driven by the space-time white noise:
\begin{equation}\label{eq:128}
\frac{\partial U}{\partial t}(t,x)-\frac{\partial ^2 U}{\partial
x^2} (t,x)=b(U(t,x))+\dot W(t,x),\quad (t,x)\in [0,T]\times [0,1],
 \end{equation}
where $T>0$ and $b$ is a globally Lipschitz function. We impose some
initial datum and Dirichlet boundary conditions. In Theorem
\ref{thm:heat} below, we will prove that the random field solution
$U$ of \eqref{eq:128} can be approximated in law, in the space of
continuous functions, by a sequence of random fields
$\{U_\ep\}_{\ep}$, where $U_\ep$ is the mild solution to a
stochastic heat equation like \eqref{eq:128} but driven by either
the real or imaginary part of the noise $X_\ep$. This result
provides an example of a kind of weak continuity phenomenon in the
path space, where convergence in law of the noisy inputs implies
convergence in law of the corresponding solutions. Another example
of this fact was provided by Walsh in \cite{walsh-neural}, where a
parabolic stochastic partial differential equation was used to model
a discontinuous neurophysiological phenomenon.

The proof of Theorem \ref{thm:heat} will follow from \cite[Thm.
1.4]{BJQ}. More precisely, Theorem 1.4 of \cite{BJQ} establishes
sufficient conditions on a family of random fields that approximate
the Brownian sheet (in some sense) under which the solutions of
\eqref{eq:128} driven by this family converges in law, in the space
of continuous functions, to the random field $U$. We refer to
Section \ref{sec:heat} for the precise statement of the
above-mentioned conditions. In \cite{BJQ}, the authors apply their
main result to two important families of random fields that
approximate the Brownian sheet: the Donsker kernels in the plane and
the Kac-Stroock processes, where the latter are defined by
\[
 \theta_n(t,x):=n\sqrt{tx}\, (-1)^{N\left(\sqrt{n}\, t,\sqrt{n}\, x\right)},
\]
where $N$ denotes a standard Poisson process in the plane (indeed,
this case corresponds to \eqref{eq:888}). As it will be exhibited in
Section \ref{sec:heat}, the proof of Theorem \ref{thm:heat} is
strongly based on the treatment of the Kac-Stroock processes in
\cite{BJQ} (see Section 4 therein), and also on some technical
estimates contained in the proof of the tightness result given in
Proposition \ref{prop:2} of the present paper.

Eventually, we note that the kind of convergence results that are obtained in the present paper
assure that the
limit processes, which in our case correspond to the complex Brownian sheet and the solution to the
stochastic heat equation, are robust when used as models in
practical situations. Moreover, the obtained results provide expressions that can be
useful to study simulations of these limit processes.

The paper is organized as follows. Section \ref{sec:prel} contains
some preliminaries on two-parameter random fields and the definition
of L\'evy sheet. Section \ref{sec:tightness} is devoted to prove
that the family of laws of $(X_\ep)_{\ep>0}$ is tight in the space
of complex-valued continuous functions. The limit identification is
addressed in Section \ref{sec:iden}. Finally, the result on weak
convergence for the stochastic heat equation is obtained in Section
\ref{sec:heat}.


\section{Preliminaries}
\label{sec:prel}

Let $(\Omega, \F,\mathbb{P})$ be a complete probability space. We will use some notation introduced by Cairoli and Walsh in \cite{CW75}.
Namely, let $\{\F_{s,t};\, (s,t)\in[0,S]\times[0,T]\}$ be
a family of sub-$\sigma$-algebras of $\F$ satisfying:
\begin{itemize}
\item[(i)]$\F_{s,t}\subset\F_{s',t'}$, for all $s\le s'$ and $t\le t'$.
\item[(ii)] All zero sets of $\F$ are contained in $\F_{0,0}$ .
\item[(iii)] For any $z\in[0,S]\times[0,T]$, $\F_z=\cap_{z<z'}\F_{z'}$, where  $z=(s,t)<z'=(s',t')$ denotes the partial order in $[0,S]\times[0,T]$,
which means that $s<s'$ and $t<t'$.
\end{itemize}

If $(s,t)<(s',t')$ and $Y$ denotes any random field defined in $[0,S]\times[0,T]$, the increment of $Y$ on the rectangle
$[(s,t),(s',t')]$ is defined by
\[
 \Delta_{s,t} Y(s',t'):= Y(s',t')-Y(s,t')-Y(s',t)+Y(s,t).
\]
An adapted process $\{Y(s,t);\, (s,t)\in[0,S]\times[0,T]\}$ with respect to the filtration
$\{\F_{s,t};\, (s,t)\in[0,S]\times[0,T]\}$ is called a martingale if $\mathbb{E}[|Y(s,t)|]<\infty$ for all $(s,t)\in[0,S]\times[0,T]$  and
$$
\mathbb{E}[\Delta_{s,t}Y(s',t')|\F_{s,t}]=0, \quad \text{for all} \quad  (s,t)<(s',t').
$$
It will be called a strong martingale if $\mathbb{E}[|Y(s,t)|]<\infty$ for all
$(s,t)\in[0,S]\times[0,T]$, $Y(s,0)=Y(0,t)=0$ for all $s,t$  and
$$
\mathbb{E}[\Delta_{s,t}Y (s',t')|\F_{S,t}\vee\F_{s,T}]=0, \quad \text{for all} \quad (s,t)<(s',t').
$$

We recall that a Brownian sheet is an adapted process $\{W(s,t);
\,(s,t)\in[0,S]\times[0,T]\}$ such that $W(s,0)=W(0,t)=0$ $\P$-a.s.,
the increment $\Delta_{s,t}W(s',t')$ is independent of $\F_{S,t}
\vee \F_{s,T}$, for all $(s,t)<(s',t')$, and it is normally
distributed with mean zero and variance $(s'-s)(t'-t)$. If no
filtration is specified, we will consider the one generated by the
process itself, namely $\F^W:=\sigma\{W(s,t);\,
(s,t)\in[0,S]\times[0,T]\}$ (conveniently completed).

A Lévy sheet is defined as follows. In general, if $Q$ is any
rectangle in $\R_+^2$ and $Y$ any random field in $\R_+^2$, we will
also denote by $\Delta_Q Y$ the increment of $Y$ on $Q$. It is
well-known that, for any nonnegative definite function $\Psi$ in
$\R$, there exists a real-valued random field $L=\{L(s,t); \,
s,t\geq 0\}$ such that
\begin{itemize}
\item[(i)] For any family of disjoint rectangles $Q_1,\dots, Q_n$ in $\R_+^2$, the increments  $\Delta_{Q_1}L,\dots, \Delta_{Q_n}L$ are
independent random variables.
\item[(ii)] For any rectangle $Q$ in $\R_+^2$, the characteristic function of the increment $\Delta_Q L$ is given by
\begin{equation}
\mathbb{E}\left[e^{i \xi \Delta_{Q}L}\right]=
e^{-\lambda(Q)\Psi(\xi)}, \quad \xi\in \R, \label{eq:999}
\end{equation}
where $\lambda$ denotes the Lebesgue measure on $\R_+^2$.
\end{itemize}

\begin{definition}\label{def:1}
A random field  $L=\{L(s,t); \, s,t\geq 0\}$  taking values in $\R$ that is continuous in probability
and satisfies the above conditions (i) and (ii) is called a Lévy sheet with exponent $\Psi$.
\end{definition}

By the Lévy-Khintchine formula, we have
$$
\Psi(\xi) = i a \xi+ \frac{1}{2}\sigma^2  \xi^2 + \int_{\R} \left[ 1- e^{i\xi x} + \frac{i x \xi}{1+|x|^2} \right]\eta(dx), \quad \xi\in \R,
$$
where $a\in\mathbb{R}$, $\sigma\ge0$ and $\eta$ is the corresponding Lévy measure, that is a Borel measure on $\R\setminus \{0\}$ that
satisfies
\[
 \int_{\R} \frac{|x|^2}{1+|x|^2} \eta(dx)<\infty.
\]
We write $\Psi(\xi)=a(\xi) + ib(\xi)$, where
$$
a(\xi):=\frac{1}{2}\sigma^2\xi^2 + \int_{\mathbb{R}}[1-\cos(\xi x)]\eta(dx),
$$
and
$$
b(\xi):=a\xi + \int_{\mathbb{R}}\left[\frac{x \xi}{1+|x|^2} - \sin(\xi x) \right] \eta(dx).
$$
Observe that $a(\xi)\geq 0$ and, if $\xi\neq 0$, $a(\xi)>0$ whenever
$\sigma>0$ and/or $\eta$ is nontrivial.


\section{Tightness}
\label{sec:tightness}

This section is devoted to prove that the family of probability laws
of $\{X_\ep\}_{\ep>0}$ is tight in $\C([0,S]\times[0,T];\bC)$. This
will be a consequence of the next result and the tightness criterion
\cite[Thm. 3]{BW} (see also \cite{centsov}), taking in the account
that our processes $X_\ep$ vanish on the axes.

\begin{proposition}\label{prop:2}
 Let $\{X_\ep\}_{\ep>0}$ be the family of random fields defined by \eqref{eq:1}.
 There exists a positive constant $C$ such that, for all $(0,0)\leq (s,t)<(s',t')\leq (S,T)$,
$$
\sup_{\ep>0} \mathbb{E}\left[\left|\Delta_{s,t}X_\ep(s',t')\right|^4\right]\le C(s'-s)^2(t'-t)^2.
$$
This implies that the the family of probability laws of
$(X_\ep)_{\ep>0}$ is tight in $\C([0,S]\times[0,T];\bC)$.
\end{proposition}

\begin{proof} By definition of $X_\ep$ and the properties of the modulus $|\cdot|$, we have
\begin{align*}
& \mathbb{E}\left[\left|\Delta_{s,t}X_\ep(s',t')\right|^4\right] \\
& \; =\ep^4 K^4 \mathbb{E} \left[ \left|\int_{\frac{t}{\ep}}^{\frac{t'}{\ep}}
\int_{\frac{s}{\ep}}^{\frac{s'}{\ep}}\sqrt{xy} \, \{\cos(\theta L(x,y))+i\sin(\theta L(x,y))\}dxdy \right|^4\right]\\
&\; =\ep^4 K^4\mathbb{E}\left[ \left|\int_{\frac{t}{\ep}}^{\frac{t'}{\ep}}
\int_{\frac{s}{\ep}}^{\frac{s'}{\ep}}\sqrt{xy}\,  e^{i\theta L(x,y)} dxdy\right|^4\right]\\
&\; =\ep^4K^4\mathbb{E}\left[\left(
\int_{\frac{t}{\ep}}^{\frac{t'}{\ep}}\int_{\frac{s}{\ep}}^{\frac{s'}{\ep}}\sqrt{x_1y_1}\,
e^{i\theta L(x_1,y_1)} dx_1dy_1
\int_{\frac{t}{\ep}}^{\frac{t'}{\ep}}
\int_{\frac{s}{\ep}}^{\frac{s'}{\ep}}\sqrt{x_2y_2}\, e^{-i\theta
L(x_2,y_2)}dx_2dy_2
\right)^2\right]\\
&\; =\ep^4 K^4 \mathbb{E}\left[\left(\int_{\frac{t}{\ep}}^{\frac{t'}{\ep}}
\int_{\frac{s}{\ep}}^{\frac{s'}{\ep}}\int_{\frac{t}{\ep}}^{\frac{t'}{\ep}}
\int_{\frac{s}{\ep}}^{\frac{s'}{\ep}}
\sqrt{x_1x_2y_1y_2}\, e^{i\theta(L(x_2,y_2)-L(x_1,y_1))} dx_1dx_2dy_1dy_2\right)^2\right]\\
&\; =\ep^4 K^4 \mathbb{E} \left[\int_{\frac{t}{\ep}}^{\frac{t'}{\ep}}
\int_{\frac{s}{\ep}}^{\frac{s'}{\ep}} \dots \int_{\frac{t}{\ep}}^{\frac{t'}{\ep}}
\int_{\frac{s}{\ep}}^{\frac{s'}{\ep}}
\sqrt{x_1x_2x_3x_4}\sqrt{y_1y_2y_3y_4}\right.\\
&\qquad \qquad  \quad \times e^{i\theta(
L(x_4,y_4)-L(x_3,y_3)+L(x_2,y_2)-L(x_1,y_1))} dx_1\dots
dx_4dy_1\dots dy_4\Bigg].
\end{align*}
Taking into account that we can write $e^{i\theta \sum_{j=1}^4(-1)^j
L(x_j,y_j)}= e^{i\theta\sum_{j=1}^4(-1)^j\Delta_{0,0}L(x_j,y_j)}$,
we obtain
\begin{align*}
\mathbb{E}\left[\left|\Delta_{s,t}X_\ep(s',t')\right|^4\right] &
=\ep^4 K^4 \int_{\frac{t}{\ep}}^{\frac{t'}{\ep}} \int_{\frac{s}{\ep}}^{\frac{s'}{\ep}}\dots
\int_{\frac{t}{\ep}}^{\frac{t'}{\ep}} \int_{\frac{s}{\ep}}^{\frac{s'}{\ep}}
\sqrt{x_1x_2x_3x_4}\sqrt{y_1y_2y_3y_4}\\
& \qquad \times \mathbb{E}\left[e^{i\theta
\sum_{j=1}^4(-1)^j\Delta_{0,0}L(x_j,y_j)}\right] dx_1\dots dx_4
dy_1\dots dy_4.
\end{align*}
In order to estimate the expectation inside the above term, we need to consider all
24 possible orders of the $x$-variables and all 24 possible orders of the $y$-variables.
Altogether, this amounts to take into account 576 possibilities. Let $\mathcal{P}_4$ be the
group of permutations of degree 4. Then,
\begin{align}
& \mathbb{E}\left[\left|\Delta_{s,t}X_\ep(s',t')\right|^4\right] \nonumber \\
& \;= \sum_{\sigma\in \mathcal{P}_4}\sum_{\beta\in
\mathcal{P}_4}\varepsilon^4K^4\int_{\frac{t}{\varepsilon}}^{\frac{t'}{\varepsilon}}\int_{\frac{s}{\varepsilon}}^{\frac{s'}{\varepsilon}}\cdots
\int_{\frac{t}{\varepsilon}}^{\frac{t'}{\varepsilon}}\int_{\frac{s}{\varepsilon}}^{\frac{s'}{\varepsilon}}
\sqrt{x_1x_2x_3x_4}\sqrt{y_1y_2y_3y_4} \, \E\left[ e^{i\theta\sum_
{j=1}^4(-1)^j\Delta_{0,0}L(x_j,y_j)} \right]\nonumber \\
& \qquad \times I_{\{x_{\sigma(1)}<x_{\sigma(2)}<x_{\sigma(3)}<x_{\sigma(4)}\}}
I_{\{y_{\beta(1)}<y_{\beta(2)}<y_{\beta(3)}<y_{\beta(4)}\}}dx_1\dots
dx_4dy_1\dots dy_4\nonumber \\
& \; \leq \sum_{\sigma\in \mathcal{P}_4}\sum_{\beta\in
\mathcal{P}_4}\varepsilon^4K^4\int_{\frac{t}{\varepsilon}}^{\frac{t'}{\varepsilon}}\int_{\frac{s}{\varepsilon}}^{\frac{s'}{\varepsilon}}\cdots
\int_{\frac{t}{\varepsilon}}^{\frac{t'}{\varepsilon}}\int_{\frac{s}{\varepsilon}}^{\frac{s'}{\varepsilon}}
\sqrt{x_1x_2x_3x_4}\sqrt{y_1y_2y_3y_4} \left|\E\left[
e^{i\theta\sum_
{j=1}^4(-1)^j\Delta_{0,0}L(x_j,y_j)} \right]\right|\nonumber \\
& \qquad \times
I_{\{x_{\sigma(1)}<x_{\sigma(2)}<x_{\sigma(3)}<x_{\sigma(4)}\}}
I_{\{y_{\beta(1)}<y_{\beta(2)}<y_{\beta(3)}<y_{\beta(4)}\}}dx_1\dots
dx_4dy_1\dots dy_4.
\label{eq:4}
\end{align}
At this point, we observe that the geometric structure of the resulting increments of $L$ in the
expression $\sum_{j=1}^4(-1)^j\Delta_{0,0}L(x_j,y_j)$ of any of the 576 possibilities
corresponds to one of the 24 cases drawn in Figure \ref{fig:1}; we note that
the latter corresponds to all 24 possible orders of the $x$-variables with
$y_1<y_2<y_3<y_4$. In each of these 24 possible structures, the corresponding increments of $L$
turn out to be multiplied by $c_1\in\{-1,1\}$ in the black regions, while they
are multiplied by $c_2\in\{-2,0,2\}$ in the white regions.

Let us now fix two permutations $\sigma,\beta\in
\mathcal{P}_4$, and we will focus on the term
\begin{equation}
\left|\E\left[ e^{i\theta\sum_{j=1}^4(-1)^j\Delta_{0,0}L(x_j,y_j)}
\right]\right| \,
I_{\{x_{\sigma(1)}<x_{\sigma(2)}<x_{\sigma(3)}<x_{\sigma(4)}\}}
I_{\{y_{\beta(1)}<y_{\beta(2)}<y_{\beta(3)}<y_{\beta(4)}\}}.
\label{eq:2}
\end{equation}
We perform a change of variables in such a way that, making a harmless abuse of notation and
using again the same one for the new variables, we have
$x_1<x_2<x_3<x_4$ and $y_1<y_2<y_3<y_4$.

On the other hand, if we denote by $Q$ the region of Figure \ref{fig:1} corresponding the above fixed variables order,
we know that $Q$ can be decomposed as a union of black rectangles and white
rectangles. More precisely, we can write
$$Q=\left( \cup_k \, \bar{Q}_k \right) \cup  \big( \cup_l\tilde{Q}_l\big),$$
where the increments $\Delta_{\bar{Q}_k} L$ are multiplied by
$c_1^k\in\{-1,1\}$ and $\Delta_{\tilde{Q}_l} L$ are multiplied by
$c_2^l\in\{-2,0,2\}$. Hence, expression \eqref{eq:2} is given by

\begin{eqnarray}
&& \left|\E\left[ e^{i\theta(\sum_k c_1^k\Delta_{\bar{Q}_k}L+\sum_l
c_2^l\Delta_{\tilde{Q}_l}L)} \right]\right|
I_{\{x_{1}<x_{2}<x_{3}<x_{4}\}}
I_{\{y_{1}<y_{2}<y_{3}<y_{4}\}} \nonumber \\
&=&\left| \E \left[ e^{i\theta \sum_k c_1^k\Delta_{\bar{Q}_k}L}
\right] \E\left[ e^{i\theta\sum_l c_2^l\Delta_{\tilde{Q}_l}L}
\right]\right|
I_{\{x_{1}<x_{2}<x_{3}<x_{4}\}} I_{\{y_{1}<y_{2}<y_{3}<y_{4}\}}\nonumber \\
&=& \left| \E \left[ e^{i\theta \sum_k c_1^k\Delta_{\bar{Q}_k}L}
\right]\right| \times \left|\E\left[ e^{i\theta\sum_l
c_2^l\Delta_{\tilde{Q}_l}L} \right]\right|
I_{\{x_{1}<x_{2}<x_{3}<x_{4}\}} I_{\{y_{1}<y_{2}<y_{3}<y_{4}\}}\nonumber\\
&=&\prod_k \left|e^{-\lambda(\bar{Q}_k)\Psi(c_1^k\theta)}\right|
\prod_l \left|e^{-\lambda(\tilde{Q}_l)\Psi(c_2^l\theta)}\right|
I_{\{x_{1}<x_{2}<x_{3}<x_{4}\}} I_{\{y_{1}<y_{2}<y_{3}<y_{4}\}}\nonumber\\
&=&\prod_k e^{-\lambda(\bar{Q}_k)a(c_1^k\theta)}\prod_l
e^{-\lambda(\tilde{Q}_l)a(c_2^l\theta)}
I_{\{x_{1}<x_{2}<x_{3}<x_{4}\}} I_{\{y_{1}<y_{2}<y_{3}<y_{4}\}}\nonumber\\
&\leq&\prod_k e^{-\lambda(\bar{Q}_k)a(\theta)}
I_{\{x_{1}<x_{2}<x_{3}<x_{4}\}} I_{\{y_{1}<y_{2}<y_{3}<y_{4}\}}\nonumber\\
&=& e^{-\lambda(\bar{Q})a(\theta)} I_{\{x_{1}<x_{2}<x_{3}<x_{4}\}}
I_{\{y_{1}<y_{2}<y_{3}<y_{4}\}}, \label{eq:3}
\end{eqnarray}
where $\bar{Q}:=\cup_k \bar{Q}_k$. In the above computations, we
have used that the real part of the Lévy exponent $\Psi$ is a
nonnegative function and satisfies $a(-\theta)=a(\theta)$. We remark
that, independently of the constants $c_1^k$ and $c_2^l$, we have
obtained an estimated of \eqref{eq:2} which only involves the black
regions multiplied by 1. Recall that $\lambda$ denotes the Lebesgue
measure on $\R^2$.


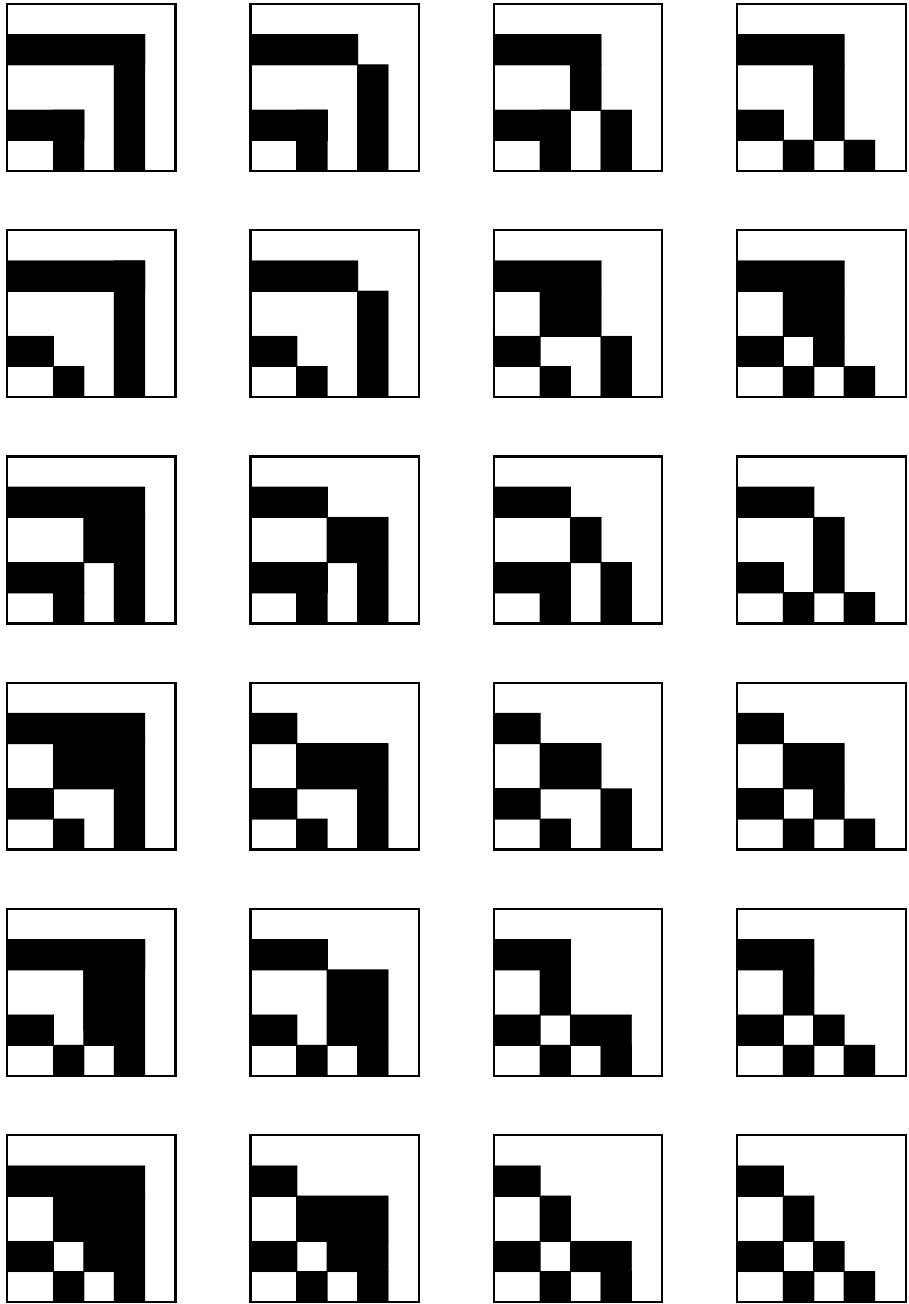
\begin{figure}
\begin{center}
\unitlength = 1mm
\begin{picture}(126,180)
\thinlines
\drawframebox{15.0}{165.0}{22.0}{22.0}{}
\drawframebox{47.0}{165.0}{22.0}{22.0}{}
\drawframebox{79.0}{165.0}{22.0}{22.0}{}
\drawframebox{111.0}{165.0}{22.0}{22.0}{}
\drawframebox{15.0}{135.0}{22.0}{22.0}{}
\drawframebox{47.0}{135.0}{22.0}{22.0}{}
\drawframebox{79.0}{135.0}{22.0}{22.0}{}
\drawframebox{111.0}{135.0}{22.0}{22.0}{}
\drawframebox{15.0}{105.0}{22.0}{22.0}{}
\drawframebox{47.0}{105.0}{22.0}{22.0}{}
\drawframebox{79.0}{105.0}{22.0}{22.0}{}
\drawframebox{111.0}{105.0}{22.0}{22.0}{}
\drawframebox{15.0}{75.0}{22.0}{22.0}{}
\drawframebox{47.0}{75.0}{22.0}{22.0}{}
\drawframebox{79.0}{75.0}{22.0}{22.0}{}
\drawframebox{111.0}{75.0}{22.0}{22.0}{}
\drawframebox{15.0}{45.0}{22.0}{22.0}{}
\drawframebox{47.0}{45.0}{22.0}{22.0}{}
\drawframebox{79.0}{45.0}{22.0}{22.0}{}
\drawframebox{111.0}{45.0}{22.0}{22.0}{}
\drawframebox{15.0}{15.0}{22.0}{22.0}{}
\drawframebox{47.0}{15.0}{22.0}{22.0}{}
\drawframebox{79.0}{15.0}{22.0}{22.0}{}
\drawframebox{111.0}{15.0}{22.0}{22.0}{}
\drawshadebox{4.0}{13.0}{22.0}{168.0}{170.0}{172.0}{}{1.0}
\drawshadebox{18.0}{20.0}{22.0}{154.0}{163.0}{172.0}{}{1.0}
\drawshadebox{4.0}{9.0}{14.0}{158.0}{160.0}{162.0}{}{1.0}
\drawshadebox{10.0}{12.0}{14.0}{154.0}{158.0}{162.0}{}{1.0}
\drawshadebox{36.0}{43.0}{50.0}{168.0}{170.0}{172.0}{}{1.0}
\drawshadebox{50.0}{52.0}{54.0}{154.0}{161.0}{168.0}{}{1.0}
\drawshadebox{36.0}{41.0}{46.0}{158.0}{160.0}{162.0}{}{1.0}
\drawshadebox{42.0}{44.0}{46.0}{154.0}{158.0}{162.0}{}{1.0}
\drawshadebox{68.0}{75.0}{82.0}{168.0}{170.0}{172.0}{}{1.0}
\drawshadebox{68.0}{73.0}{78.0}{158.0}{160.0}{162.0}{}{1.0}
\drawshadebox{74.0}{76.0}{78.0}{154.0}{158.0}{162.0}{}{1.0}
\drawshadebox{82.0}{84.0}{86.0}{154.0}{158.0}{162.0}{}{1.0}
\drawshadebox{78.0}{80.0}{82.0}{162.0}{165.0}{168.0}{}{1.0}
\drawshadebox{100.0}{107.0}{114.0}{168.0}{170.0}{172.0}{}{1.0}
\drawshadebox{100.0}{103.0}{106.0}{158.0}{160.0}{162.0}{}{1.0}
\drawshadebox{106.0}{108.0}{110.0}{154.0}{156.0}{158.0}{}{1.0}
\drawshadebox{114.0}{116.0}{118.0}{154.0}{156.0}{158.0}{}{1.0}
\drawshadebox{110.0}{112.0}{114.0}{158.0}{163.0}{168.0}{}{1.0}
\drawshadebox{4.0}{13.0}{22.0}{138.0}{140.0}{142.0}{}{1.0}
\drawshadebox{18.0}{20.0}{22.0}{124.0}{133.0}{142.0}{}{1.0}
\drawshadebox{4.0}{7.0}{10.0}{128.0}{130.0}{132.0}{}{1.0}
\drawshadebox{10.0}{12.0}{14.0}{124.0}{126.0}{128.0}{}{1.0}
\drawshadebox{36.0}{43.0}{50.0}{138.0}{140.0}{142.0}{}{1.0}
\drawshadebox{50.0}{52.0}{54.0}{124.0}{131.0}{138.0}{}{1.0}
\drawshadebox{36.0}{39.0}{42.0}{128.0}{130.0}{132.0}{}{1.0}
\drawshadebox{42.0}{44.0}{46.0}{124.0}{126.0}{128.0}{}{1.0}
\drawshadebox{68.0}{75.0}{82.0}{138.0}{140.0}{142.0}{}{1.0}
\drawshadebox{74.0}{78.0}{82.0}{132.0}{135.0}{138.0}{}{1.0}
\drawshadebox{68.0}{71.0}{74.0}{128.0}{130.0}{132.0}{}{1.0}
\drawshadebox{74.0}{76.0}{78.0}{124.0}{126.0}{128.0}{}{1.0}
\drawshadebox{82.0}{84.0}{86.0}{124.0}{128.0}{132.0}{}{1.0}
\drawshadebox{100.0}{107.0}{114.0}{138.0}{140.0}{142.0}{}{1.0}
\drawshadebox{106.0}{110.0}{114.0}{132.0}{135.0}{138.0}{}{1.0}
\drawshadebox{110.0}{112.0}{114.0}{128.0}{130.0}{132.0}{}{1.0}
\drawshadebox{100.0}{103.0}{106.0}{128.0}{130.0}{132.0}{}{1.0}
\drawshadebox{106.0}{108.0}{110.0}{124.0}{126.0}{128.0}{}{1.0}
\drawshadebox{114.0}{116.0}{118.0}{124.0}{126.0}{128.0}{}{1.0}
\drawshadebox{4.0}{13.0}{22.0}{108.0}{110.0}{112.0}{}{1.0}
\drawshadebox{18.0}{20.0}{22.0}{94.0}{101.0}{108.0}{}{1.0}
\drawshadebox{4.0}{9.0}{14.0}{98.0}{100.0}{102.0}{}{1.0}
\drawshadebox{10.0}{12.0}{14.0}{94.0}{96.0}{98.0}{}{1.0}
\drawshadebox{14.0}{16.0}{18.0}{102.0}{105.0}{108.0}{}{1.0}
\drawshadebox{36.0}{41.0}{46.0}{108.0}{110.0}{112.0}{}{1.0}
\drawshadebox{36.0}{41.0}{46.0}{98.0}{100.0}{102.0}{}{1.0}
\drawshadebox{42.0}{44.0}{46.0}{94.0}{96.0}{98.0}{}{1.0}
\drawshadebox{46.0}{50.0}{54.0}{102.0}{105.0}{108.0}{}{1.0}
\drawshadebox{50.0}{52.0}{54.0}{94.0}{98.0}{102.0}{}{1.0}
\drawshadebox{68.0}{73.0}{78.0}{108.0}{110.0}{112.0}{}{1.0}
\drawshadebox{68.0}{73.0}{78.0}{98.0}{100.0}{102.0}{}{1.0}
\drawshadebox{74.0}{76.0}{78.0}{94.0}{96.0}{98.0}{}{1.0}
\drawshadebox{78.0}{80.0}{82.0}{102.0}{105.0}{108.0}{}{1.0}
\drawshadebox{82.0}{84.0}{86.0}{94.0}{98.0}{102.0}{}{1.0}
\drawshadebox{100.0}{105.0}{110.0}{108.0}{110.0}{112.0}{}{1.0}
\drawshadebox{110.0}{112.0}{114.0}{98.0}{103.0}{108.0}{}{1.0}
\drawshadebox{114.0}{116.0}{118.0}{94.0}{96.0}{98.0}{}{1.0}
\drawshadebox{100.0}{103.0}{106.0}{98.0}{100.0}{102.0}{}{1.0}
\drawshadebox{106.0}{108.0}{110.0}{94.0}{96.0}{98.0}{}{1.0}
\drawshadebox{4.0}{13.0}{22.0}{78.0}{80.0}{82.0}{}{1.0}
\drawshadebox{18.0}{20.0}{22.0}{64.0}{71.0}{78.0}{}{1.0}
\drawshadebox{10.0}{14.0}{18.0}{72.0}{75.0}{78.0}{}{1.0}
\drawshadebox{4.0}{7.0}{10.0}{68.0}{70.0}{72.0}{}{1.0}
\drawshadebox{10.0}{12.0}{14.0}{64.0}{66.0}{68.0}{}{1.0}
\drawshadebox{36.0}{39.0}{42.0}{78.0}{80.0}{82.0}{}{1.0}
\drawshadebox{42.0}{48.0}{54.0}{72.0}{75.0}{78.0}{}{1.0}
\drawshadebox{36.0}{39.0}{42.0}{68.0}{70.0}{72.0}{}{1.0}
\drawshadebox{42.0}{44.0}{46.0}{64.0}{66.0}{68.0}{}{1.0}
\drawshadebox{50.0}{52.0}{54.0}{64.0}{68.0}{72.0}{}{1.0}
\drawshadebox{68.0}{71.0}{74.0}{78.0}{80.0}{82.0}{}{1.0}
\drawshadebox{74.0}{78.0}{82.0}{72.0}{75.0}{78.0}{}{1.0}
\drawshadebox{68.0}{71.0}{74.0}{68.0}{70.0}{72.0}{}{1.0}
\drawshadebox{74.0}{76.0}{78.0}{64.0}{66.0}{68.0}{}{1.0}
\drawshadebox{82.0}{84.0}{86.0}{64.0}{68.0}{72.0}{}{1.0}
\drawshadebox{100.0}{103.0}{106.0}{78.0}{80.0}{82.0}{}{1.0}
\drawshadebox{106.0}{110.0}{114.0}{72.0}{75.0}{78.0}{}{1.0}
\drawshadebox{110.0}{112.0}{114.0}{68.0}{70.0}{72.0}{}{1.0}
\drawshadebox{114.0}{116.0}{118.0}{64.0}{66.0}{68.0}{}{1.0}
\drawshadebox{100.0}{103.0}{106.0}{68.0}{70.0}{72.0}{}{1.0}
\drawshadebox{106.0}{108.0}{110.0}{64.0}{66.0}{68.0}{}{1.0}
\drawshadebox{4.0}{13.0}{22.0}{48.0}{50.0}{52.0}{}{1.0}
\drawshadebox{18.0}{20.0}{22.0}{34.0}{43.0}{52.0}{}{1.0}
\drawshadebox{14.0}{16.0}{18.0}{40.0}{44.0}{48.0}{}{1.0}
\drawshadebox{14.0}{16.0}{18.0}{38.0}{43.0}{48.0}{}{1.0}
\drawshadebox{4.0}{7.0}{10.0}{38.0}{40.0}{42.0}{}{1.0}
\drawshadebox{10.0}{12.0}{14.0}{34.0}{36.0}{38.0}{}{1.0}
\drawshadebox{36.0}{41.0}{46.0}{48.0}{50.0}{52.0}{}{1.0}
\drawshadebox{46.0}{50.0}{54.0}{38.0}{43.0}{48.0}{}{1.0}
\drawshadebox{50.0}{52.0}{54.0}{34.0}{36.0}{38.0}{}{1.0}
\drawshadebox{36.0}{39.0}{42.0}{38.0}{40.0}{42.0}{}{1.0}
\drawshadebox{42.0}{44.0}{46.0}{34.0}{36.0}{38.0}{}{1.0}
\drawshadebox{68.0}{73.0}{78.0}{48.0}{50.0}{52.0}{}{1.0}
\drawshadebox{74.0}{76.0}{78.0}{42.0}{45.0}{48.0}{}{1.0}
\drawshadebox{68.0}{71.0}{74.0}{38.0}{40.0}{42.0}{}{1.0}
\drawshadebox{78.0}{82.0}{86.0}{38.0}{40.0}{42.0}{}{1.0}
\drawshadebox{82.0}{84.0}{86.0}{34.0}{36.0}{38.0}{}{1.0}
\drawshadebox{74.0}{76.0}{78.0}{34.0}{36.0}{38.0}{}{1.0}
\drawshadebox{100.0}{105.0}{110.0}{48.0}{50.0}{52.0}{}{1.0}
\drawshadebox{106.0}{108.0}{110.0}{42.0}{45.0}{48.0}{}{1.0}
\drawshadebox{100.0}{103.0}{106.0}{38.0}{40.0}{42.0}{}{1.0}
\drawshadebox{110.0}{112.0}{114.0}{38.0}{40.0}{42.0}{}{1.0}
\drawshadebox{106.0}{108.0}{110.0}{34.0}{36.0}{38.0}{}{1.0}
\drawshadebox{114.0}{116.0}{118.0}{34.0}{36.0}{38.0}{}{1.0}
\drawshadebox{4.0}{13.0}{22.0}{18.0}{20.0}{22.0}{}{1.0}
\drawshadebox{18.0}{20.0}{22.0}{4.0}{11.0}{18.0}{}{1.0}
\drawshadebox{10.0}{14.0}{18.0}{12.0}{15.0}{18.0}{}{1.0}
\drawshadebox{14.0}{16.0}{18.0}{8.0}{10.0}{12.0}{}{1.0}
\drawshadebox{4.0}{7.0}{10.0}{8.0}{10.0}{12.0}{}{1.0}
\drawshadebox{10.0}{12.0}{14.0}{4.0}{6.0}{8.0}{}{1.0}
\drawshadebox{36.0}{39.0}{42.0}{18.0}{20.0}{22.0}{}{1.0}
\drawshadebox{42.0}{48.0}{54.0}{12.0}{15.0}{18.0}{}{1.0}
\drawshadebox{46.0}{50.0}{54.0}{8.0}{10.0}{12.0}{}{1.0}
\drawshadebox{50.0}{52.0}{54.0}{4.0}{6.0}{8.0}{}{1.0}
\drawshadebox{36.0}{39.0}{42.0}{8.0}{10.0}{12.0}{}{1.0}
\drawshadebox{42.0}{44.0}{46.0}{4.0}{6.0}{8.0}{}{1.0}
\drawshadebox{68.0}{71.0}{74.0}{18.0}{20.0}{22.0}{}{1.0}
\drawshadebox{74.0}{76.0}{78.0}{12.0}{15.0}{18.0}{}{1.0}
\drawshadebox{78.0}{82.0}{86.0}{8.0}{10.0}{12.0}{}{1.0}
\drawshadebox{82.0}{84.0}{86.0}{4.0}{6.0}{8.0}{}{1.0}
\drawshadebox{68.0}{71.0}{74.0}{8.0}{10.0}{12.0}{}{1.0}
\drawshadebox{74.0}{76.0}{78.0}{4.0}{6.0}{8.0}{}{1.0}
\drawshadebox{100.0}{103.0}{106.0}{18.0}{20.0}{22.0}{}{1.0}
\drawshadebox{106.0}{108.0}{110.0}{12.0}{15.0}{18.0}{}{1.0}
\drawshadebox{110.0}{112.0}{114.0}{8.0}{10.0}{12.0}{}{1.0}
\drawshadebox{114.0}{116.0}{118.0}{4.0}{6.0}{8.0}{}{1.0}
\drawshadebox{100.0}{103.0}{106.0}{8.0}{10.0}{12.0}{}{1.0}
\drawshadebox{106.0}{108.0}{110.0}{4.0}{6.0}{8.0}{}{1.0}
\end{picture}
\caption{Each square represents the rectangle $[(s,t),(s',t')]$. 
Regions corresponding to $\sum_{j=1}^4(-1)^j\Delta_{0,0}L(x_j,y_j)$,
for all possible 24 orders of the $x$-variables and $y_1<y_2<y_3<y_4$, are drawn in each square. 
Black areas are regions where the corresponding increment of $L$ appears 
an odd number of times. Note that, indeed, all areas are extended up to the plane axes.
}
\label{fig:1}
\end{center}
\end{figure}

%

Taking into account estimate \eqref{eq:3}, it is readily checked
that, among all 24 possibilities drawn in Figure \ref{fig:1}, it
suffices to deal with 4 of these cases (see Figure \ref{fig:2}).
This is because, in the rest of the cases, the area of $\bar{Q}$ is
greater than or equal to the corresponding one of one of these 4
possibilities. We point out that, in some cases, one needs to apply
a symmetry argument. For instance, the analysis of the figure in the
fourth row and third column of Figure \ref{fig:1}, by symmetry
between the $x$ and $y$ coordinates, is equivalent to that of the
figure in the third row and fourth column. Thus, since in
\eqref{eq:3} the area of $\bar{Q}$ appears with a negative sign, we
can focus only on the cases of Figure \ref{fig:2}.

\begin{figure}
\begin{center}
\unitlength = 0.9mm
\begin{picture}(92,52)
\thinlines
\drawpath{4.0}{8.0}{44.0}{8.0}
\drawpath{48.0}{8.0}{88.0}{8.0}
\drawpath{8.0}{4.0}{8.0}{44.0}
\drawpath{52.0}{4.0}{52.0}{44.0}
\drawframebox{32.0}{32.0}{20.0}{20.0}{}
\drawframebox{76.0}{32.0}{20.0}{20.0}{}
\drawshadebox{8.0}{21.0}{34.0}{34.0}{36.0}{38.0}{}{1.0}
\drawshadebox{34.0}{36.0}{38.0}{8.0}{21.0}{34.0}{}{1.0}
\drawshadebox{8.0}{17.0}{26.0}{26.0}{28.0}{30.0}{}{1.0}
\drawshadebox{26.0}{28.0}{30.0}{8.0}{17.0}{26.0}{}{1.0}
\drawshadebox{52.0}{61.0}{70.0}{34.0}{36.0}{38.0}{}{1.0}
\drawshadebox{70.0}{72.0}{74.0}{30.0}{32.0}{34.0}{}{1.0}
\drawshadebox{52.0}{61.0}{70.0}{26.0}{28.0}{30.0}{}{1.0}
\drawshadebox{70.0}{72.0}{74.0}{8.0}{17.0}{26.0}{}{1.0}
\drawshadebox{74.0}{76.0}{78.0}{26.0}{28.0}{30.0}{}{1.0}
\drawshadebox{78.0}{80.0}{82.0}{8.0}{17.0}{26.0}{}{1.0}
\drawdotline{22.0}{22.0}{22.0}{8.0}
\drawdotline{42.0}{22.0}{42.0}{8.0}
\drawdotline{66.0}{22.0}{66.0}{8.0}
\drawdotline{86.0}{22.0}{86.0}{8.0}
\drawdotline{22.0}{42.0}{8.0}{42.0}
\drawdotline{22.0}{22.0}{8.0}{22.0}
\drawdotline{66.0}{22.0}{52.0}{22.0}
\drawdotline{66.0}{42.0}{52.0}{42.0}
\drawcenteredtext{26.0}{6.0}{$x_{1}$}
\drawcenteredtext{30.0}{6.0}{$x_{2}$}
\drawcenteredtext{34.0}{6.0}{$x_{3}$}
\drawcenteredtext{38.0}{6.0}{$x_{4}$}
\drawcenteredtext{22.0}{6.0}{$s$}
\drawcenteredtext{42.0}{6.0}{$s'$}
\drawcenteredtext{66.0}{6.0}{$s$}
\drawcenteredtext{70.0}{6.0}{$x_{1}$}
\drawcenteredtext{74.0}{6.0}{$x_{2}$}
\drawcenteredtext{78.0}{6.0}{$x_{3}$}
\drawcenteredtext{82.0}{6.0}{$x_{4}$}
\drawcenteredtext{86.0}{6.0}{$s'$}
\drawcenteredtext{6.0}{22.0}{$t$}
\drawcenteredtext{6.0}{26.0}{$y_{1}$}
\drawcenteredtext{50.0}{26.0}{$y_{1}$}
\drawcenteredtext{50.0}{22.0}{$t$}
\drawcenteredtext{6.0}{30.0}{$y_{2}$}
\drawcenteredtext{50.0}{30.0}{$y_{2}$}
\drawcenteredtext{50.0}{34.0}{$y_{3}$}
\drawcenteredtext{6.0}{34.0}{$y_{3}$}
\drawcenteredtext{6.0}{38.0}{$y_{4}$}
\drawcenteredtext{50.0}{38.0}{$y_{4}$}
\drawcenteredtext{6.0}{42.0}{$t'$}
\drawcenteredtext{50.0}{42.0}{$t'$}
\drawcenteredtext{4.0}{48.0}{i)}
\drawcenteredtext{48.0}{48.0}{ii)}
\end{picture}
\begin{picture}(92,52)(88,0)
\thinlines
\drawpath{92.0}{8.0}{132.0}{8.0}
\drawpath{136.0}{8.0}{176.0}{8.0}
\drawpath{96.0}{4.0}{96.0}{44.0}
\drawpath{140.0}{4.0}{140.0}{44.0}
\drawframebox{120.0}{32.0}{20.0}{20.0}{}
\drawframebox{164.0}{32.0}{20.0}{20.0}{}
\drawshadebox{96.0}{107.0}{118.0}{34.0}{36.0}{38.0}{}{1.0}
\drawshadebox{118.0}{120.0}{122.0}{26.0}{30.0}{34.0}{}{1.0}
\drawshadebox{96.0}{105.0}{114.0}{26.0}{28.0}{30.0}{}{1.0}
\drawshadebox{114.0}{116.0}{118.0}{8.0}{17.0}{26.0}{}{1.0}
\drawshadebox{122.0}{124.0}{126.0}{8.0}{17.0}{26.0}{}{1.0}
\drawshadebox{140.0}{151.0}{162.0}{34.0}{36.0}{38.0}{}{1.0}
\drawshadebox{162.0}{164.0}{166.0}{30.0}{32.0}{34.0}{}{1.0}
\drawshadebox{140.0}{151.0}{162.0}{26.0}{28.0}{30.0}{}{1.0}
\drawshadebox{158.0}{160.0}{162.0}{8.0}{17.0}{26.0}{}{1.0}
\drawshadebox{166.0}{168.0}{170.0}{8.0}{19.0}{30.0}{}{1.0}
\drawdotline{110.0}{22.0}{110.0}{8.0}
\drawdotline{130.0}{22.0}{130.0}{8.0}
\drawdotline{154.0}{22.0}{154.0}{8.0}
\drawdotline{174.0}{22.0}{174.0}{8.0}
\drawdotline{110.0}{42.0}{96.0}{42.0}
\drawdotline{110.0}{22.0}{96.0}{22.0}
\drawdotline{154.0}{22.0}{140.0}{22.0}
\drawdotline{154.0}{42.0}{140.0}{42.0}
\drawcenteredtext{114.0}{6.0}{$x_{1}$}
\drawcenteredtext{118.0}{6.0}{$x_{2}$}
\drawcenteredtext{122.0}{6.0}{$x_{3}$}
\drawcenteredtext{126.0}{6.0}{$x_{4}$}
\drawcenteredtext{130.0}{6.0}{$s'$}
\drawcenteredtext{110.0}{6.0}{$s$}
\drawcenteredtext{154.0}{6.0}{$s$}
\drawcenteredtext{158.0}{6.0}{$x_{1}$}
\drawcenteredtext{162.0}{6.0}{$x_{2}$}
\drawcenteredtext{166.0}{6.0}{$x_{3}$}
\drawcenteredtext{170.0}{6.0}{$x_{4}$}
\drawcenteredtext{174.0}{6.0}{$s'$}
\drawcenteredtext{94.0}{26.0}{$y_{1}$}
\drawcenteredtext{138.0}{26.0}{$y_{1}$}
\drawcenteredtext{138.0}{22.0}{$t$}
\drawcenteredtext{94.0}{22.0}{$t$}
\drawcenteredtext{94.0}{30.0}{$y_{2}$}
\drawcenteredtext{138.0}{30.0}{$y_{2}$}
\drawcenteredtext{138.0}{34.0}{$y_{3}$}
\drawcenteredtext{94.0}{34.0}{$y_{3}$}
\drawcenteredtext{94.0}{38.0}{$y_{4}$}
\drawcenteredtext{138.0}{38.0}{$y_{4}$}
\drawcenteredtext{94.0}{42.0}{$t'$}
\drawcenteredtext{138.0}{42.0}{$t'$}
\drawcenteredtext{92.0}{48.0}{iii)}
\drawcenteredtext{136.0}{48.0}{iv)}
\end{picture}
\caption{The 4 relevant cases of Figure \ref{fig:1}}
\label{fig:2}
\end{center}
\end{figure}
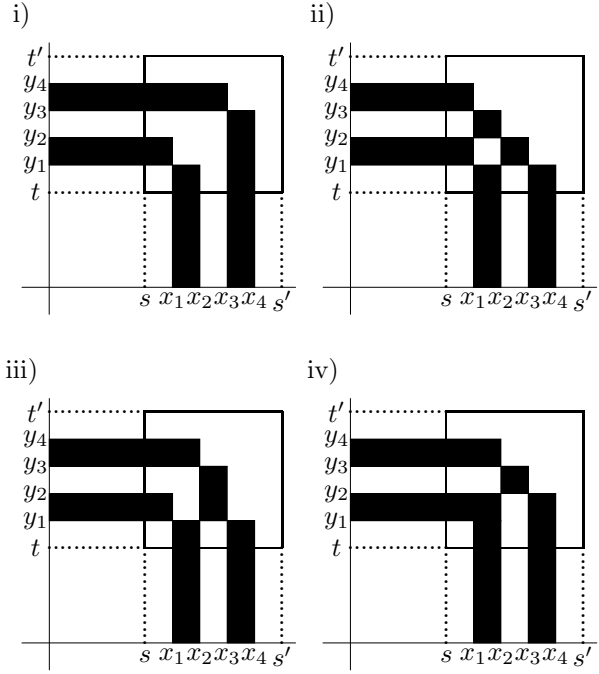


Let us start tackling the case corresponding to i) in Figure \ref{fig:2}. That is,
by estimates \eqref{eq:4} and \eqref{eq:3}, we need to find suitable upper bounds
of the term

\begin{eqnarray}
&&\varepsilon^{4}\int_{I}\int_{J}
\sqrt{x_{1}x_{2}x_{3}x_{4}}\sqrt{y_{1}y_{2}y_{3}y_{4}}\,
\exp[-a(\theta)[(x_{4}-x_{3})y_{3}+(y_{4}-y_{3})x_{3}\cr \nonumber
&&\qquad \qquad +(x_{2}-x_{1})y_{1}+(y_{2}-y_{1})x_{1}]] \,
dx_{1}...dx_{4}dy_{1}...dy_{4}, \label{eq:5}
\end{eqnarray}
where $J=\{\frac{s}{\varepsilon}\leq x_{1}\leq x_{2}\leq x_{3}\leq
x_{4}\leq \frac{s'}{\varepsilon}\}$ and $I=\{\frac{t}{\varepsilon}\leq
y_{1}\leq y_{2}\leq y_{3}\leq y_{4}\leq \frac{t'}{\varepsilon}\}$.

First, estimate $x_{4}$ and $y_{4}$ in the square roots above by $\frac{s'}{\varepsilon}$ and
$\frac{t'}{\varepsilon}$, respectively, and then integrate with respect to theses
two variables. The resulting expression can be easily bounded by, up to some positive constant,
$$\sqrt{s'}\sqrt{t'}
\varepsilon^{3}\int_{I_{1}}\int_{J_{1}}
\frac{\sqrt{x_{1}x_{2}x_{3}}\sqrt{y_{1}y_{2}y_{3}}}{x_{3}y_{3}}
\exp\big[-a(\theta)[(x_{2}-x_{1})y_{1}+(y_{2}-y_{1})x_{1}]\big]
dx_{1}...dx_{3}dy_{1}...dy_{3},$$ where
$J_{1}=\{\frac{s}{\varepsilon}\leq x_{1}\leq x_{2}\leq x_{3}\leq
\frac{s'}{\varepsilon}\}$ and $I_{1}=\{\frac{t}{\varepsilon}\leq
y_{1}\leq y_{2}\leq y_{3}\leq \frac{t'}{\varepsilon}\}$. Now, we
estimate $x_{2}$ and $y_{2}$ by $\frac{s'}{\varepsilon}$ and
$\frac{t'}{\varepsilon}$, respectively, inside the square roots, and
then integrate with respect to $x_{2}$ and $y_{2}$. Hence, up to some
constant, we obtain an estimate for \eqref{eq:5} of the form
\begin{align*}
&s't'\varepsilon^{2}\int_{\frac{s}{\varepsilon}}^{\frac{s'}{\varepsilon}}\frac{1}{\sqrt{x_{1}}}\,dx_{1}
\int_{\frac{s}{\varepsilon}}^{\frac{s'}{\varepsilon}}\frac{1}{\sqrt{x_{3}}}\,dx_{3}
\int_{\frac{t}{\varepsilon}}^{\frac{t'}{\varepsilon}}\frac{1}{\sqrt{y_{1}}}\,dy_{1}
\int_{\frac{t}{\varepsilon}}^{\frac{t'}{\varepsilon}}\frac{1}{\sqrt{y_{3}}}\, dy_{3}\\
& \quad \qquad    =C\big[\sqrt{s'}(\sqrt{s'}-\sqrt{s})\big]^{2}\big[\sqrt{t'}(\sqrt{t'}-\sqrt{t})\big]^{2}\\
& \quad \qquad \leq C(s'-s)^{2}(t'-t)^{2}.
\end{align*}
This concludes the analysis of i) in Figure \ref{fig:2}.

In the remaining three cases, the above-used argument does not directly work.
Instead, we will add some small area in the corresponding drawing in such
a way that we will be able to argue similarly as in case i).
We remark that some of the integrand's estimates that will be obtained in the sequel
will hold everywhere except of a zero Lebesgue measure set of $\mathbb R^{8}$.

Let us start with the analysis of the integral corresponding to ii). We need to
bound the following term:
$$\varepsilon^{4}\int_{I}\int_{J}
\sqrt{x_{1}x_{2}x_{3}x_{4}}\sqrt{y_{1}y_{2}y_{3}y_{4}} \,
e^{-\lambda(\bar{Q})a(\theta)} dx_{1}...dx_{4}dy_{1}...dy_{4},$$
where $J=\{\frac{s}{\varepsilon}\leq x_{1}\leq x_{2}\leq x_{3}\leq
x_{4}\leq \frac{s'}{\varepsilon}\}$, $I=\{\frac{t}{\varepsilon}\leq
y_{1}\leq y_{2}\leq y_{3}\leq y_{4}\leq \frac{t'}{\varepsilon}\}$,
and $\bar{Q}$ is the union of black rectangles corresponding to the
case ii). Note that $A:=\lambda(\bar{Q})$ is given by
$$A=(x_{4}-x_{3})y_{1}+(y_{4}-y_{3})x_{1}
+(x_{2}-x_{1})y_{1}+(y_{2}-y_{1})x_{1}
+(x_{3}-x_{2})(y_{2}-y_{1})+(y_{3}-y_{2})(x_{2}-x_{1}).$$
We split the above integral into two terms:
\begin{align}
& \varepsilon^{4}\int_{I}\int_{J}
\sqrt{x_{1}x_{2}x_{3}x_{4}}\sqrt{y_{1}y_{2}y_{3}y_{4}} \,
e^{-a(\theta)A}\, I_{\{A\geq 2(x_{2}-x_{1})(y_{2}-y_{1})\}}
dx_{1}...dx_{4}dy_{1}...dy_{4}\nonumber \\
& \qquad +\varepsilon^{4}\int_{I}\int_{J}
\sqrt{x_{1}x_{2}x_{3}x_{4}}\sqrt{y_{1}y_{2}y_{3}y_{4}} \,
e^{-a(\theta)A}\,  I_{\{A< 2(x_{2}-x_{1})(y_{2}-y_{1})\}}
dx_{1}...dx_{4}dy_{1}...dy_{4}. \label{eq:7}
\end{align}
When $A\geq 2(x_{2}-x_{1})(y_{2}-y_{1})$, we have
\begin{eqnarray*}
-a(\theta)A&\leq& -\frac{a(\theta)}{2}A-
a(\theta)(x_{2}-x_{1})(y_{2}-y_{1})\cr & \cr
&=&-\frac{a(\theta)}{2}[(x_{4}-x_{3})y_{1}+(y_{4}-y_{3})x_{1}+(x_{2}-x_{1})y_{3}
+(y_{2}-y_{1})x_{3}].\cr
\end{eqnarray*}
Hence, the first integral in \eqref{eq:7} is less or equal than
\begin{align*}
&\varepsilon^{4}\int_{I}\int_{J}
\sqrt{x_{1}x_{2}x_{3}x_{4}}\sqrt{y_{1}y_{2}y_{3}y_{4}}
\, \exp\left\{-\frac{a(\theta)}{2}[ (x_{4}-x_{3})y_{1}+(y_{4}-y_{3})x_{1}\right. \\
&\qquad \qquad \qquad +(x_{2}-x_{1})y_{3} +(y_{2}-y_{1})x_{3}]\Big\}
dx_{1}...dx_{4}dy_{1}...dy_{4},
\end{align*}
and following the same arguments used in the case i),
this term can be estimated by $(s'-s)^{2}(t'-t)^{2}$, up to some positive constant.

On the other hand, as far as the second integral in \eqref{eq:7}
is concerned, observe that we have
\begin{align*}
&\{A< 2(x_{2}-x_{1})(y_{2}-y_{1})\}\\
&  \qquad = \{
(x_{4}-x_{3})y_{1}+(y_{4}-y_{3})x_{1}
+(y_{2}-y_{1})x_{3}+(x_{2}-x_{1})y_{3}<
4(x_{2}-x_{1})(y_{2}-y_{1})\}.
\end{align*}
In particular, in this region we have
$$\frac{1}{4}y_{3}<(y_{2}-y_{1}) \quad\textnormal{and}\quad
\frac{1}{4}x_{3}<(x_{2}-x_{1}),$$
which implies that
\begin{eqnarray*}
A&\geq& (x_{4}-x_{3})y_{1}+(y_{4}-y_{3})x_{1}+\frac{1}{4}x_{3}y_{1}+
\frac{1}{4}y_{3}x_{1}+\frac{1}{4}(x_{3}-x_{2})y_{3}+\frac{1}{4}(y_{3}-
y_{2})x_{3}\cr
& \cr
&\geq&\frac{1}{4}\big[x_{4}y_{1}+y_{4}x_{1}+(x_{3}-x_{2})y_{3}+(y_{3}-y_{2})x_{3}\big].\cr
\end{eqnarray*}
Thus, the second integral in \eqref{eq:7} can be bounded by
\begin{eqnarray*}
&&\varepsilon^{4}\int_{I}\int_{J}
\sqrt{x_{1}x_{2}x_{3}x_{4}}\sqrt{y_{1}y_{2}y_{3}y_{4}}
\\
&& \qquad \qquad \times \exp\left\{-\frac{a(\theta)}{4}[
x_{4}y_{1}+y_{4}x_{1}+(x_{3}-x_{2})y_{3} + (y_{3}-y_{2})x_{3}]
\right\} dx_{1}...dx_{4}dy_{1}...dy_{4},\cr
\end{eqnarray*}
and here again the arguments of the case i) may be applied, yielding
an estimate of the form $(s'-s)^{2}(t'-t)^{2}$, up to some positive constant.

The same idea can be used to deal with the integral corresponding to iii).
Indeed, in this case the area $A$ is given by
$$A=(x_{4}-x_{3})y_{1}+(x_{2}-x_{1})y_{1}+(x_{3}-x_{2})(y_{3}-y_{1})+
(y_{4}-y_{3})x_{2}+(y_{2}-y_{1})x_{1},$$
and here one splits the underlying integral taking into account the
regions $\{A\geq 2(x_{3}-x_{2})y_{1}\}$ and $\{A< 2(x_{3}-x_{2})y_{1}\}$.
In the former, one has
\begin{eqnarray*}
-a(\theta) A &\leq&-\frac{a(\theta)}{2} A-
a(\theta)(x_{3}-x_{2})y_{1}\cr & \cr
&=&-\frac{a(\theta)}{2}[(x_{4}-x_{1})y_{1}+(y_{2}-y_{1})x_{1}+(x_{3}-x_{2})y_{3}+
(y_{4}-y_{3})x_{2}]\cr
\end{eqnarray*}
and the desired estimated is obtained by using the
same computations as for the case i). Note that, in fact, variables
which have to be bounded and integrated with respect to are
$x_{4},\, y_{4},\, y_{2}, \, x_{3}$, following this specific order.
On the other hand, in the region $\{A< 2(x_{3}-x_{2})y_{1}\}$, we get
$$\{A< 2(x_{3}-x_{2})y_{1}\}=
\{(x_{4}-x_{1})y_{1}+(y_{2}-y_{1})x_{1}+(x_{3}-x_{2})y_{3}+
(y_{4}-y_{3})x_{2}< 4(x_{3}-x_{2})y_{1}\}.$$
So, in particular, in this region we have
$$\frac{1}{4}y_{3}<y_{1}\quad\textnormal{and}\quad
\frac{1}{4}(x_{4}-x_{1})<(x_{3}-x_{2}),$$
where we deduce
\begin{eqnarray*}
A
&\geq&\frac{1}{4}\big[
(x_{4}-x_{3})y_{3}+(x_{2}-x_{1})y_{3}+(x_{4}-x_{1})(y_{3}-y_{1})+
(y_{4}-y_{3})x_{2}+ (y_{2}-y_{1})x_{1}\big]\cr
& \cr
&\geq&\frac{1}{4}\big[
(y_{4}-y_{3})x_{1}+(x_{2}-x_{1})y_{3}+(x_{4}-x_{3})y_{1}+
(y_{2}-y_{1})x_{3}\big].\cr
\end{eqnarray*}
At this point, we can follow the arguments of the preceding cases.

Finally, it only remains to estimate the integral involving case iv) in Figure \ref{fig:2}.
In this case,
$$A =
(x_{4}-x_{3})y_{2}+(x_{2}-x_{1})y_{2}+(y_{4}-y_{3})x_{2}+
(y_{2}-y_{1})x_{1}+(x_{3}-x_{2})(y_{3}-y_{2})]$$
and the splitting regions are $\{A\geq 2(x_{3}-x_{2})y_{2}\}\cup\{A\geq2(y_{3}-y_{2})x_{2}\}$
and the corresponding complement.

In the first region, condition $A\geq 2(x_{3}-x_{2})y_{2}$ turns out to be equivalent to
$$-a(\theta) A\leq-\frac{a(\theta)}{2}[(x_{3}-x_{2})y_{3}+(y_{4}-y_{3})x_{2}+(y_{2}-y_{1})x_{1}+
(x_{4}-x_{1})y_{1}],$$
so we will be able to mimic the arguments used so far. Moreover, note that this case
is symmetric in $x$ and $y$, which implies that the computations in the
case $A\geq 2(y_{3}-y_{2})x_{2}$ will be the same just by exchanging $x$ and $y$.

As far as the case $\{A< 2(x_{3}-x_{2})y_{2}\}\cap\{A<2(y_{3}-y_{2})x_{2}\}$
is concerned, we have
\begin{align*}
&\{A< 2(x_{3}-x_{2})y_{2}\}\cap\{A<2(y_{3}-y_{2})x_{2}\}\\
& \quad =
\{(x_{4}-x_{1})y_{2}+(y_{4}-y_{3})x_{2}+(y_{2}-y_{1})x_{1}+
(x_{3}-x_{2})y_{3}\leq4(x_{3}-x_{2})y_{2}\}\\
& \quad \qquad \cap
\{(y_{4}-y_{1})x_{2}+(x_{4}-x_{3})y_{2}+(x_{2}-x_{1})y_{1}+
(y_{3}-y_{2})x_{3}\leq4(y_{3}-y_{2})x_{2}\}.
\end{align*}
In particular, one has
$$\frac{1}{4}y_{3}\leq y_{2}\quad\textnormal{and}\quad
\frac{1}{4}x_{3}\leq x_{2},$$
which implies
\begin{eqnarray*}
A&\geq&\frac{1}{4}\big[(x_{4}-x_{3})y_{3}+(x_{2}-x_{1})y_{3}+(y_{4}-y_{3})x_{3}+
(y_{2}-y_{1})x_{1}\big]\cr
& \cr
&\geq&\frac{1}{4}\big[(x_{4}-x_{3})y_{3}+(y_{4}-y_{3})x_{3}+(x_{2}-x_{1})y_{1}+
(y_{2}-y_{1})x_{1}\big].\cr
\end{eqnarray*}
One can conclude the proof by following the same arguments as in the preceding cases.
\end{proof}


\section{Limit identification}
\label{sec:iden}

Let $\{\mathbb{P}_\varepsilon\}_{\varepsilon>0}$ be the family of
probability laws in $\C([0,S]\times [0,T];\bC)$ corresponding to
$\{X_\varepsilon\}_{\varepsilon>0}$. By Proposition \ref{prop:2},
there exists a subsequence $\{\mathbb{P}_{\varepsilon_n}\}_{n\geq
1}$ of $\{\mathbb{P}_\varepsilon\}_{\varepsilon>0}$ converging, in
the weak sense in the space $\C([0,S]\times [0,T];\bC)$, to some
probability measure $\P$. This section is devoted to prove that $\P$
is the law of a complex Brownian sheet, that is a random field whose
real and imaginary parts are independent Brownian sheets.

We will use the following characterization of the complex Brownian
sheet. It is an adaptation of the characterization of the
(real-valued) Brownian sheet given in \cite[Thm. 4.1]{BJ97} (other
characterizations of Brownian sheet can be found, e.g., in
\cite{FN,Tudor}).

\begin{theorem}\label{thm:3}
 Let $X=\{X(s,t);\, (s,t)\in [0,S]\times [0,T]\}$ be a complex-valued and
 continuous process such that $X(s,0)=X(0,t)=0$
 for all $s\in [0,S]$ and $t\in [0,T]$. We write $X=X^1+ i X^2$.
 Let $\{\F_{s,t};\, (s,t)\in [0,S]\times [0,T]\}$ be the natural
 filtration associated to $X$. Then, the following statements are equivalent:
 \begin{itemize}
  \item[(i)] $X$ is a complex Brownian sheet.
  \item[(ii)] $X^1$ and $X^2$ are strong martingales and,
  for all $(0,0)<(s,t)\leq (s',t')\leq (S,T)$ and $i=1,2$,
  \begin{align}
   & \qquad \E\big[\big(\Delta_{s,t} X^i(s',t')\big)^2|\F_{s,T}\big]   =
  (s'-s)(t'-t), \label{eq:111}\\
    & \E\big[\big(\Delta_{s,t}X^1(s',t')\big)\big(\Delta_{s,t}X^2(s',t')\big)
    | \F_{s,T}\big]  =0 \label{eq:8888}.
  \end{align}
 \end{itemize}
\end{theorem}

\begin{proof}
    It is obvious that (i) implies  (ii) because, for all $(0,0)<(s,t)\leq (s',t')\leq (S,T)$ and
    $i=1,2$, it holds
    $$\E[\Delta_{s,t}X^i(s',t')|\F_{S,t}\lor\F_{s,T}]= \E[\Delta_{s,t}X^i(s',t')]=0$$
    and
    \begin{align*}
        \E\big[\big(\Delta_{s,t}X^i(s',t')\big)^{2}|\F_{s,T}\big] & =
        \E\big[ \E\big[\big(\Delta_{s,t}X^i(s',t')\big)^{2}|\F_{S,t}\lor\F_{s,T}\big]|
        \F_{s,T}\big]\\
         & =
        (s'-s)(t'-t).
    \end{align*}

    We will prove now that (ii) implies (i). First, we check that
    $X^1$ and $X^2$ define (real-valued) Brownian sheets. We will
    only
    write the proof for $X^1$, because for $X^2$ it is exactly the
    same.

    Fix $(0,0)<(s,t)<(s',t')\leq (S,T)$ and define the process
    $Y=\{Y_u, \, u\in [s,S]\}$ by $Y_{u}:=\Delta_{s,t}X^1 (u,t')$ (note that, indeed, it
    does not depend on $s'$). This process is a
    martingale with respect to the filtration  $\{\F_{u,T},\, u\in[s,S]
    \}$, since it is adapted to it and, for any $s\leq u<u'\leq S$,
    \begin{align*}
        \E[Y_{u'}-Y_{u}|\F_{u,T}] & =
        \E[\Delta_{u,t}X^1(u',t')|\F_{u,T}]\\
        & =  \E\big[ \E[\Delta_{u,t}X^1(u',t')|\F_{u,T}\lor\F_{S,t}]|\F_{u,T}\big]
         =0,
    \end{align*}
because $X^1$ is a strong martingale, by hypothesis. On the other
hand, note that, by \eqref{eq:111},
    $$
    \E\big[(Y_{u'}-Y_{u})^{2}|\F_{u,T}\big] = \E\big[(\Delta_{u,t}X^1 (u',t'))^{2}|\F_{u,T}\big]
    =(u'-u)(t'-t),
    $$
which implies that the quadratic variation of the process $Y$ is
$\langle Y \rangle_u = (u-s)(t'-t)$. Hence, by L\'evy's
characterization theorem, we obtain that  $Y=\{Y_{u}, \, u\in
[s,S]\}$ defines a Brownian motion with respect to the filtration
$\{\F_{u,T},\, u\in [s,S]\}$ and with variance function $u\mapsto
(u-s)(t'-t)$.

Thus, the increments $\Delta_{s,t}X^1 (s',t')=Y_{s'}-Y_{s}$ are
normally distributed random variables with mean zero and variance
$(s'-s)(t'-t)$. Note that, so far, we are assuming that $s,t>0$. In
the case $t=0$ (the case $s=0$ can be argued in the same way), we
consider the increments $\Delta_{s,\varepsilon}X^1(s',t')$, with
$\varepsilon>0$. We know that they are Gaussian random variables,
and they converge, as $\varepsilon$ tends to zero, to
$\Delta_{s,0}X^1(s',t')$, which will thus also be a centered
Gaussian random variable with variance equal to the corresponding
limit of variances, that is $(s'-s)t'$. Finally, observe that the
(rectangular) increments of $X^1$ are independent since those of $Y$
are. Therefore, $X^1$ defines a Brownian sheet.

As we mentioned before, the same argument proves that $X^2$ also
defines a Brownian sheet. Hence, in order to conclude the proof, it
remains to check that $X^1$ and $X^2$ are independent. By
\eqref{eq:8888}, it suffices to verify that any linear combination
of $X^1$ and $X^2$ defines a Gaussian random variable. For this, we
use a similar argument as above, as follows. Let $a,b\in \mathbb{R}$
(not both equal to zero) and $(0,0)<(s,t)<(s',t')\leq (S,T)$, and
define
\[
Z_u:= a \Delta_{s,t} X^1(u,t')+ b \Delta_{s,t} X^2(u,t'), \quad u\in
[s,S].
\]
Then, one verifies that the process $Z=\{Z_u,\, u\in [s,S]\}$ is a
martingale with respect to the filtration $\{\F_{u,T},\,
u\in[s,S]\}$ and, owing to \eqref{eq:8888}, it has quadratic
variation $\langle Z\rangle_u=(u-s)(t'-t)(a^2+b^2)$. L\'evy's
characterization theorem implies that $Z$ is a Brownian motion with
variance function $u\mapsto (u-s)(t'-t)(a^2+b^2)$. In particular,
$Z_u$ defines a Gaussian random variable, for all $u\in [s,S]$. As
we did above, this fact can be extended to the case where $s=0$
and/or $t=0$. This concludes the proof.
\end{proof}

Owing to Theorem \ref{thm:3} and Proposition \ref{prop:2}, the
following two propositions will guarantee the validity of (almost
all) the statement of Theorem \ref{thm:main}.

\begin{proposition}\label{prop:5}
Recall that $\P$ denotes the weak limit in $\C([0,S]\times
[0,T];\bC)$ of a converging subsequence of the family
$\{\mathbb{P}_\varepsilon\}_{\varepsilon>0}$. Let $X=\{X(s,t);\,
(s,t)\in [0,S]\times [0,T]\}$ be the corresponding (complex-valued)
canonical process and $\{\F_{s,t};\, (s,t)\in [0,S]\times [0,T]\}$
its associated natural filtration. Then, the real and imaginary
parts of $X$ define strong martingales under the probability $\P$.
\end{proposition}

\begin{proposition}\label{prop:6}
Let $X$ be the canonical process defined in the previous
proposition. Then, for all $(0,0)< (s,t)\leq (s',t')\leq (S,T)$, it
holds:
\[
\E_\P \big[ \big(\Delta_{s,t} \mathsf{Re}(X)(s',t')\big)^2
|\F_{s,T}\big]=(s'-s)(t'-t),
\]
and
\[
\E_\P \big[ \big(\Delta_{s,t} \mathsf{Im}(X)(s',t')\big)^2
|\F_{s,T}\big]=(s'-s)(t'-t),
\]
\end{proposition}

In the sequel, we will need to compute some limits, and we will use
l'H\^opital's rule in its usual form. However, sometimes it may be
long and tedious to check whether we are under the hypotheses of
l'H\^opital's theorem. The following lemma is a version of this
result which makes things easier. Its proof is an easy application
of the mean value theorem.
\begin{lemma}
    Suppose that $f:[M,\infty)\rightarrow \mathbb R$ is a derivable
    function,
    such that $f'$ is continuous on $[M,\infty)$, $M\ge 0$,
    and assume that
    $$\lim_{u\to\infty}f'(u)=a<\infty.$$
    Then,
    $$\lim_{u\to\infty}\frac{f(u)}{u}=a.$$
\end{lemma}

The proof of Proposition \ref{prop:5} is based on the following
lemma.

\begin{lemma}\label{lem:9}
Let $X_\ep=\{X_\ep(s,t);\, (s,t)\in [0,S]\times [0,T]\}$ be the
(complex-valued) random field defined in \eqref{eq:1} and
$\{\F^\ep_{s,t};\, (s,t)\in [0,S]\times [0,T]\}$ its natural
filtration. Then, for all $(0,0)< (s,t)\leq (s',t')\leq (S,T)$,
\begin{equation}\label{eq:11}
\lim_{\ep\to 0} \E [\Delta_{s,t} X_\ep(s',t')
|\F^\ep_{S,t}\vee\F^\ep_{s,T}]=0,
\end{equation}
where the limit is understood in $L^2(\Omega)$.
\end{lemma}

\begin{proof}
We will use the notation $Y_\ep:=\E [\Delta_{s,t} X_\ep(s',t')
|\F_{S,t}\vee\F_{s,T}]$. First, note that we can write
\[
\Delta_{s,t} X_\ep(s',t')= \ep
K\int_{\frac{t}{\ep}}^{\frac{t'}{\ep}}\int_{\frac{s}{\ep}}^{\frac{s'}{\ep}}\sqrt{xy}
\, e^{i\theta
\left(L(\frac{s}{\ep},y)+L(x,\frac{t}{\ep})-L(\frac{s}{\ep},\frac{t}{\ep})+
\Delta_{\frac{s}{\ep},\frac{t}{\ep}}L(x,y)\right)} dxdy.
\]
Thus, we have
\begin{align*}
Y_\ep & = \ep K\int_{\frac{t}{\ep}}^{\frac{t'}{\ep}}
\int_{\frac{s}{\ep}}^{\frac{s'}{\ep}}\sqrt{xy} \, e^{i\theta
\left(L(\frac{s}{\ep},y)+L(x,\frac{t}{\ep})-L(\frac{s}{\ep},\frac{t}{\ep})\right)}
\mathbb{E}\left[e^{i\theta \Delta_{\frac{s}{\ep},\frac{t}{\ep}}L(x,y)}\right] dxdy\\
&=\ep
K\int_{\frac{t}{\ep}}^{\frac{t'}{\ep}}\int_{\frac{s}{\ep}}^{\frac{s'}{\ep}}
\sqrt{xy} \, e^{i\theta
\left(L(\frac{s}{\ep},y)+L(x,\frac{t}{\ep})-L(\frac{s}{\ep},\frac{t}{\ep})\right)}
e^{-\Psi(\theta)\left(x-\frac{s}{\ep}\right)\left(y-\frac{t}{\ep}\right)}dxdy
\end{align*}
and, recalling that $\Psi(\xi)=a(\xi)+ib(\xi)$ is the L\'evy
exponent,
\begin{align*}
\E\big[|Y_\ep|^2\big] & = \ep^2 K^2
\int_{\frac{t}{\ep}}^{\frac{t'}{\ep}}
\int_{\frac{t}{\ep}}^{\frac{t'}{\ep}}
\int_{\frac{s}{\ep}}^{\frac{s'}{\ep}}
\int_{\frac{s}{\ep}}^{\frac{s'}{\ep}} \sqrt{x_1 x_2} \sqrt{y_1 y_2}
\, \E\left[ e^{i\theta \left(L(\frac s\ep,y_1) + L(x_1,\frac t\ep) -
L(\frac s\ep,y_2) -  L(x_2,\frac t\ep) \right)}\right]\\
& \qquad \times  e^{-a(\theta) \left((x_1-\frac s\ep)(y_1-\frac
t\ep) + (x_2-\frac s\ep)(y_2-\frac t\ep)\right)}  e^{-i b(\theta)
\left((x_1-\frac s\ep)(y_1-\frac t\ep) - (x_2-\frac s\ep)(y_2-\frac
t\ep)\right)} dx_1 dx_ 2 dy_1 dy_2.
\end{align*}
In order to compute the expectation inside the integral, we take
into account the possible orders of $x_1,x_2$ and $y_1,y_2$,
respectively, which amounts to consider 4 possibilities. Then, in
each case we express the exponent of the complex exponential in the
above expectation as a suitable combination of rectangular
increments of $L$, so that we can compute the corresponding
expectation thanks to \eqref{eq:999}. In the resulting four terms,
we get rid of the complex exponentials simply by applying the
modulus' triangle inequality and putting the modulus inside the
integrals. Using this procedure, we end up with
\begin{equation}\label{eq:13}
\E\big[|Y_\ep|^2\big]\leq 2(I_1+I_2),
\end{equation}
where
\begin{align*}
I_1 & = \ep^2 K^2 \int_{\frac{t}{\ep}}^{\frac{t'}{\ep}}
\int_{\frac{t}{\ep}}^{y_2} \int_{\frac{s}{\ep}}^{\frac{s'}{\ep}}
\int_{\frac{s}{\ep}}^{x_2} \sqrt{x_1 x_2} \sqrt{y_1 y_2} \,
e^{-a(\theta) \left( (x_2-x_1)\frac t\ep + (y_2-y_1)\frac s\ep \right)} \\
& \qquad \qquad \times e^{-a(\theta) \left((x_1-\frac
s\ep)(y_1-\frac t\ep) + (x_2-\frac s\ep)(y_2-\frac t\ep)\right)}
dx_1 dx_ 2 dy_1 dy_2
\end{align*}
and
\begin{align*}
I_2 & = \ep^2 K^2 \int_{\frac{t}{\ep}}^{\frac{t'}{\ep}}
\int_{\frac{t}{\ep}}^{y_2} \int_{\frac{s}{\ep}}^{\frac{s'}{\ep}}
\int_{\frac{s}{\ep}}^{x_1} \sqrt{x_1 x_2} \sqrt{y_1 y_2} \,
e^{-a(\theta) \left( (x_1-x_2)\frac t\ep + (y_2-y_1)\frac s\ep \right)} \\
& \qquad \qquad \times e^{-a(\theta) \left((x_1-\frac
s\ep)(y_1-\frac t\ep) + (x_2-\frac s\ep)(y_2-\frac t\ep)\right)}
dx_2 dx_ 1 dy_1 dy_2.
\end{align*}
Applying Fubini theorem, one easily verifies that both $I_1$ and
$I_2$ can be bounded by the term
\[
I = \ep^2 K^2 \int_{\frac{t}{\ep}}^{\frac{t'}{\ep}}
\int_{y_1}^{\frac{t'}{\ep}} \int_{\frac{s}{\ep}}^{\frac{s'}{\ep}}
\int_{\frac{s}{\ep}}^{x_2} \sqrt{x_1 x_2} \sqrt{y_1 y_2} \,
e^{-a(\theta) \left( (x_2-x_1)\frac t\ep + (y_2-y_1)\frac s\ep
\right)}   e^{-a(\theta) (x_2-\frac s\ep)(y_1-\frac t\ep)} dx_1 dx_
2 dy_2 dy_1.
\]
This implies that $\E\big[|Y_\ep|^2\big]\leq 4 I$. Let us check that
$I$ converges to zero as $\ep\to 0$. Indeed, bounding $\sqrt{x_1}$
by $\sqrt{x_2}$ and $\sqrt{y_2}$ by $\sqrt{\frac{t'}{\ep}}$, then
integrating with respect to $x_1$ and $y_2$, and finally applying a
change of variables, we can infer that
\begin{align*}
I & \leq C \ep \sqrt{\ep} \int_{\frac{t}{\ep}}^{\frac{t'}{\ep}}
\int_{y_1}^{\frac{t'}{\ep}} \int_{\frac{s}{\ep}}^{\frac{s'}{\ep}}
\int_{\frac{s}{\ep}}^{x_2} x_2 \sqrt{y_1} \, e^{-a(\theta) \left(
(x_2-x_1)\frac t\ep + (y_2-y_1)\frac s\ep \right)} e^{-a(\theta)
(x_2-\frac s\ep)(y_1-\frac t\ep)} dx_1 dx_ 2 dy_2 dy_1\\
& \leq C \ep^3 \sqrt{\ep} \int_{\frac{t}{\ep}}^{\frac{t'}{\ep}}
\int_{\frac{s}{\ep}}^{\frac{s'}{\ep}}
 x_2 \sqrt{y_1} \, e^{-a(\theta)
(x_2-\frac s\ep)(y_1-\frac t\ep)} dx_2 dy_1\\
& = C \ep^3 \sqrt{\ep} \int_{0}^{\frac{t'-t}{\ep}}
\int_{0}^{\frac{s'-s}{\ep}}
 \left(x_2+\frac s\ep\right)  \sqrt{y_1+\frac t\ep} \; e^{-a(\theta)
x_2 y_1} dx_2 dy_1\\
& \leq C \ep^2 \int_{0}^{\frac{t'-t}{\ep}}
\int_{0}^{\frac{s'-s}{\ep}} e^{-a(\theta) x_2 y_1} dx_2 dy_1,
\end{align*}
where $C$ is some positive constant whose value may change from line
to line. The latter expression converges to zero as $\ep\to 0$. In
order to see this, e.g., one splits the integral with respect to
$x_2$ on the intervals $(0,\ep)$ and $(\ep,\frac{s'-s}{\ep})$. In
$(0,\ep)$ one bounds the exponential by 1 and then integrate, while
in $(\ep,\frac{s'-s}{\ep})$ one first integrates with respect to
$y_1$ and then with respect to $x_2$. This concludes the proof.
\end{proof}

We are now in position to prove Proposition \ref{prop:5}:

{\it{Proof of Proposition \ref{prop:5}.}} First, we check that
$X(s,0)=X(0,t)=0$, for all $(s,t)\in [0,S]\times [0,T]$. By
Skorohod's representation theorem, there exist a probability space
$(\tilde{\Omega}, \tilde{\mathcal{F}}, \tilde{\mathbb{P}})$, a
sequence $\{Y_\ep\}_{\ep>0}$ and a random variable $Y$, all of them
taking values in $\C([0,S]\times [0,T])$, such that
\begin{enumerate}
\item[(a)] For all $\ep>0$, $Y_\ep$ and $X_\ep$ have the same law,
\item[(b)] $Y$ and $X$ have the same law,
\item[(c)] $Y_\ep$ converges to $Y$, as $\ep\to 0$,
$\tilde{\mathbb{P}}$-almost surely.
\end{enumerate}
Condition (c) means that
\[
\lim_{\ep\to 0 } \sup_{(s,t)\in [0,S]\times [0,T]}
|Y_\ep(s,t)-Y(s,t)|=0, \quad \tilde{\mathbb{P}}\text{-a.s.}
\]
In particular, for any $t\in [0,T]$, $\lim_{\ep\to 0}
Y_\ep(0,t)=Y(0,t)$, $\tilde{\mathbb{P}}$-a.s. By (a), it holds
$\tilde{\mathbb{P}}(Y_\ep(0,t)=0)=\mathbb{P}_\ep(X_\ep(0,t)=0)=1$.
Hence, letting $\ep\to 0$, we get $\tilde{\mathbb{P}}(Y(0,t)=0)=1$.
By (b), we end up with $\mathbb{P}(X(0,t)=0)=1$. The same argument
let us conclude that $\mathbb{P}(X(s,0)=0)=1$, for all $s\in [0,S]$.

 The remaining of the proof is similar to that of \cite[Prop. 4.2]{BJ97}. Let $(0,0)<
(s,t)<(s',t')\leq (S,T)$. It suffices to prove that, for any $n\geq
1$ and $(s_1,t_1),\dots,(s_n,t_n)$ such that either $s_i\leq S$ and
$t_i\leq t$, or $s_i\leq s$ and $t_i\leq T$, $i=1,\dots,n$, and for
any continuous and bounded function $\varphi:\bC^n\to\R$, it holds
that
\[
\left|\E_\P [\varphi(X(s_1,t_1),\dots,X(s_n,t_n)) \Delta_{s,t}
X(s',t')]\right|=0.
\]
We recall that the notation $|z|$ stands for the modulus of $z\in
\bC$. Without any loss of generality, the converging subsequence of
probability measures to $\P$ will be simply denoted by
$\{\P_\ep\}_{\ep>0}$. Thus, by Proposition \ref{prop:2}, it suffices
to check that
\[
\lim_{\ep\to 0} \left|\E_{\P_\ep}
[\varphi(X(s_1,t_1),\dots,X(s_n,t_n)) \Delta_{s,t}
X(s',t')]\right|=0.
\]
For this, we recall that, as in the statement of Lemma \ref{lem:9},
$\{\F^\ep_{s,t};\, (s,t)\in [0,S]\times [0,T]\}$ is the natural
filtration associated to the (complex-valued) random field $X^\ep$
introduced in \eqref{eq:1}. Then, we can argue as follows:
\begin{align*}
& \left|\E_{\P_\ep} [\varphi(X(s_1,t_1),\dots,X(s_n,t_n))
\Delta_{s,t} X(s',t')]\right| \\
& \qquad = \left|\E [\varphi(X_\ep(s_1,t_1),\dots,X_\ep(s_n,t_n))
\Delta_{s,t}
X_\ep(s',t')]\right|\\
& \qquad \leq \left|\E
\left[\varphi(X_\ep(s_1,t_1),\dots,X_\ep(s_n,t_n)) \, \E
[\Delta_{s,t}
X_\ep(s',t')|\F^\ep_{S,t}\vee\F^\ep_{s,T}]\right]\right|\\
& \qquad \leq C \left( \E\left[ \left|\E [\Delta_{s,t}
X_\ep(s',t')|\F^\ep_{S,t}\vee\F^\ep_{s,T}]\right|^2\right]\right)^\frac12.
\end{align*}
The latter term converges to zero as $\ep\to 0$, by Lemma
\ref{lem:9}. \qed

\medskip

In order to prove Proposition \ref{prop:6}, we need two auxiliary results. The first one is the following.

\begin{lemma}\label{lem:25}
 For any $(0,0)\leq (s,t)\leq (s',t')\leq (S,T)$, it holds:
 \[
  \lim_{\ep\to 0} \E\big[ |\Delta_{s,t} X_\ep(s',t')|^2\big] = 2(s'-s)(t'-t).
 \]
\end{lemma}

\begin{proof} We split the proof in three steps.

{\it{Step 1.}}
 Owing to the definition of $X_\ep$ (see \eqref{eq:1}) and applying Fubini theorem, we have
 \begin{align*}
 & \E\big[ |\Delta_{s,t} X_\ep(s',t')|^2\big]\\
 & \qquad =
  \ep^2 K^2 \int_{\frac{t}{\ep}}^{\frac{t'}{\ep}}\int_{\frac{t}{\ep}}^{\frac{t'}{\ep}}
    \int_{\frac{s}{\ep}}^{\frac{s'}{\ep}} \int_{\frac{s}{\ep}}^{\frac{s'}{\ep}}
    \sqrt{x_1x_2}\sqrt{y_1y_2}\,\, \mathbb{E}\left[ e^{i\theta \big(\Delta_{0,0}L(x_2,y_2)-\Delta_{0,0}L(x_1,y_1)\big)}
    \right]dx_1dx_2dy_1dy_2.
 \end{align*}
As in the proof of Lemma \ref{lem:9}, we need to take into account the possible orders of $x_1, x_2$ and
$y_1, y_2$, respectively. Then, applying also some suitable changes of variables, we have
\begin{align}
 & \E\big[ |\Delta_{s,t} X_\ep(s',t')|^2\big]\nonumber \\
 & \qquad =
  \ep^2 K^2 \int_{\frac{t}{\ep}}^{\frac{t'}{\ep}}
  \int_{\frac{t}{\ep}}^{y_2} \int_{\frac{s}{\ep}}^{\frac{s'}{\ep}}
  \int_{\frac{s}{\ep}}^{x_2} \sqrt{x_1 x_2 y_1 y_2}\,
  e^{-\Psi(\theta)\left((y_2-y_1)x_1+(x_2-x_1)y_2\right)}dx_1dx_2dy_1dy_2\nonumber \\
  & \qquad \quad +
  \ep^2 K^2 \int_{\frac{t}{\ep}}^{\frac{t'}{\ep}}
  \int_{\frac{t}{\ep}}^{y_2} \int_{\frac{s}{\ep}}^{\frac{s'}{\ep}}
  \int_{\frac{s}{\ep}}^{x_2} \sqrt{x_1 x_2 y_1 y_2}\,
  e^{-\Psi(\theta)(x_2-x_1)y_1} e^{-\Psi(-\theta)(y_2-y_1)x_1}dx_1dx_2dy_1dy_2\nonumber \\
  & \qquad \quad +
  \ep^2 K^2 \int_{\frac{t}{\ep}}^{\frac{t'}{\ep}}
  \int_{\frac{t}{\ep}}^{y_2} \int_{\frac{s}{\ep}}^{\frac{s'}{\ep}}
  \int_{\frac{s}{\ep}}^{x_2} \sqrt{x_1 x_2 y_1 y_2}\,
  e^{-\Psi(-\theta)(x_2-x_1)y_1} e^{-\Psi(\theta)(y_2-y_1)x_1}dx_1dx_2dy_1dy_2\nonumber \\
  & \qquad \quad +
  \ep^2 K^2 \int_{\frac{t}{\ep}}^{\frac{t'}{\ep}}
  \int_{\frac{t}{\ep}}^{y_2} \int_{\frac{s}{\ep}}^{\frac{s'}{\ep}}
  \int_{\frac{s}{\ep}}^{x_2} \sqrt{x_1 x_2 y_1 y_2}\,
  e^{-\Psi(-\theta)\left((y_2-y_1)x_1+(x_2-x_1)y_2\right)}dx_1dx_2dy_1dy_2.
  \label{eq:585}
\end{align}
Recalling that $\Psi(\theta)=a(\theta)+ib(\theta)$, where
$a(\theta)=a(-\theta)$ and $b(\theta)=-b(-\theta)$, we observe that
\begin{align*}
 & e^{-\Psi(\theta)\left((y_2-y_1)x_1+(x_2-x_1)y_2\right)}+e^{-\Psi(\theta)(x_2-x_1)y_1}e^{-\Psi(-\theta)(y_2-y_1)x_1}\\
 & \qquad \quad + e^{-\Psi(-\theta)(x_2-x_1)y_1}e^{-\Psi(\theta)(y_2-y_1)x_1}+
 e^{-\Psi(-\theta)\left((y_2-y_1)x_1+(x_2-x_1)y_2\right)}\\
 & \quad = e^{-a(\theta)\left(x_2y_2-x_1y_1\right)}2 \cos \left(b(\theta)(x_2y_2-x_1y_1)\right)\\
 & \qquad \qquad
 +e^{-a(\theta)\left((y_2-y_1)x_1+(x_2-x_1)y_1\right)} 2 \cos\left(b(\theta)((y_2-y_1)x_1{\color{red}{-}}(x_2-x_1)y_1)\right).
\end{align*}
As a consequence, we can infer that
\begin{equation}\label{eq:51}
 \E\big[ |\Delta_{s,t} X_\ep(s',t')|^2\big]= 2(I^\ep_1+I^\ep_2),
\end{equation}
where
\begin{equation}\label{eq:52}
 I^\ep_1 = \ep^2 K^2 \int_{\frac{t}{\ep}}^{\frac{t'}{\ep}}
  \int_{\frac{t}{\ep}}^{y_2} \int_{\frac{s}{\ep}}^{\frac{s'}{\ep}}
  \int_{\frac{s}{\ep}}^{x_2} \sqrt{x_1 x_2 y_1 y_2}\, e^{-a(\theta)\left(x_2y_2-x_1y_1\right)} \cos \left(b(\theta)(x_2y_2-x_1y_1)\right)
  dx_1dx_2dy_1dy_2
\end{equation}
and
\begin{align}
 I_2^\ep & = \ep^2 K^2 \int_{\frac{t}{\ep}}^{\frac{t'}{\ep}}
  \int_{\frac{t}{\ep}}^{y_2} \int_{\frac{s}{\ep}}^{\frac{s'}{\ep}}
  \int_{\frac{s}{\ep}}^{x_2} \sqrt{x_1 x_2 y_1 y_2}\,
  e^{-a(\theta)\left((y_2-y_1)x_1+(x_2-x_1)y_1\right)} \nonumber \\
  & \qquad \qquad \times \cos\left(b(\theta)((y_2-y_1)x_1+(x_2-x_1)y_1)\right)
  dx_1dx_2dy_1dy_2.
  \label{eq:23}
\end{align}

\smallskip

{\it{Step 2.}}
Let us consider the case $s=t=0$. In order to deal with $I_1^\ep$, we make the changes of variables $z_i:=x_i y_i$ and
$v_i:= \frac{\ep}{s'} x_i$, $i=1,2$, and we define $u:=\frac{s' t'}{\ep^2}$. Thus, by l'H\^opital's rule, we have
\[
 \lim_{\ep\to 0} I_1^\ep = \lim_{u\to \infty}
 s't'K^2 \int_0^1 \int_0^{v_2} \int_0^{uv_1} \frac{\sqrt{z_1uv_2}}{v_1} \,
 e^{-a(\theta)uv_2+a(\theta)z_1} \cos\left(b(\theta)(uv_2-z_1)\right) dz_1dv_1dv_2.
\]
Applying now the changes of variables $v_2':=uv_2$ and $v_1':=u v_1$, and again l'H\^opital's rule, we obtain that
the latter limit equals to
\begin{equation}\label{eq:22}
 \lim_{u\to \infty}
 s't'K^2  \int_0^u \int_0^{v_1'} \frac{\sqrt{z_1 u }}{v_1'} \,
 e^{-a(\theta)u + a(\theta)z_1} \cos\left(b(\theta)(u-z_1)\right) dz_1dv_1'.
\end{equation}
In order to compute the above limit, we use the formula
$\cos(\theta)=\frac12 (e^{i\theta}+e^{-i\theta})$. Hence, the
expression inside the limit \eqref{eq:22} can be written as the sum
$\frac12(A_u+B_u)$, where these terms are given by
\[
 A_u:= s't'K^2  \int_0^u \int_0^{v} \frac{\sqrt{z u }}{v} \,
 e^{-a(\theta)u + a(\theta)z} \, e^{i b(\theta)(u-z)} dzdv,
\]
\[
 B_u:= s't'K^2  \int_0^u \int_0^{v} \frac{\sqrt{z u }}{v} \,
 e^{-a(\theta)u + a(\theta)z} \, e^{-i b(\theta)(u-z)} dzdv.
\]
 We will only deal with $\lim_{u\to \infty} B_u$, because $\lim_{u\to \infty} A_u$ can be treated in a similar way.
 Indeed,  rewriting $B_u$ as
 \[
  B_u= s't'K^2  \frac{\int_0^u \int_0^{v} \frac{\sqrt{z}}{v} \,
 e^{(a(\theta) +ib(\theta))z}  dzdv}{u^{-\frac12} \, e^{(a(\theta)+ib(\theta))u}}
 \]
 and applying l'H\^opital's rule twice, one easily proves that
 \[
  \lim_{u\to \infty} B_u= \frac{s' t' K^2}{(a(\theta)+ib(\theta))^2}.
 \]
Similarly, one gets
\[
  \lim_{u\to \infty} A_u= \frac{s' t' K^2}{(a(\theta)-ib(\theta))^2}.
 \]
Thus,
\[
 \lim_{\ep\to 0} I_1^\ep = s' t' K^2 \frac{a(\theta)^2-b(\theta)^2}{(a(\theta)^2+b(\theta)^2)^2}.
\]

Now, we are going to compute  $\lim_{\ep\to0} I_2^\ep$. Recall that
the latter term is given in \eqref{eq:23}. The strategy that we have
followed to deal with $I_1^\ep$ cannot be applied here. More
precisely, we have not been able to compute the limit of $I_2^\ep$
directly, but we will introduce an auxiliary term which will
converge to some quantity, and we will prove that the remainder
converges to zero.

To start with, we apply the same changes of variables that we
performed for $I_1^\ep$, we set $u:=\frac{s't'}{\ep}$ and apply
l'H\^opital's rule, so $\lim_{\ep\to0}I_2^\ep$ equals to
\[
 \lim_{u\to\infty}
 K^2 s' t' \int_0^1\int_0^{v_2}\int_0^{uv_1} \sqrt{u v_2}\frac{\sqrt{z_1}}{v_1} \, e^{a(\theta)\left(2z_1-uv_1-\frac{z_1 v_2}{v_1}\right)}
 \cos\Big(b(\theta)\Big(uv_1-\frac{z_1 v_2}{v_1}\Big)\Big) dz_1 dv_1 dv_2.
\]
Next, we make the changes of variables $\bar{v}_1:=uv_1$ and $\bar{v}_2:=uv_2$ and we apply again l'H\^opital's rule. Hence, the latter limit
becomes
\[
 \lim_{u\to\infty}
 K^2 s' t' \sqrt{u}  \int_0^u \int_0^{\bar{v}_1}  \frac{\sqrt{z_1}}{\bar{v}_1} \, e^{a(\theta)\left(2z_1-\bar{v}_1-\frac{z_1 u}{\bar{v}_1}\right)}
 \cos\Big(b(\theta)\Big(\bar{v}_1-\frac{z_1 u}{\bar{v}_1}\Big)\Big) dz_1 d\bar{v}_1.
\]
Finally, performing the changes $x:=\frac{z_1}{\bar{v}_1}$ and
$y:=\frac{\bar{v}_1}{u}$, we end up with
\begin{equation*}
 \lim_{\ep\to0}I_2^\ep = \lim_{u\to\infty} C_u,
\end{equation*}
with
\begin{equation}\label{eq:30}
 C_u=K^2 s' t' u^2 \int_0^1 \int_0^1 \sqrt{xy} \, e^{a(\theta)(2xy-y-x)u}
 \cos(b(\theta)(y-x)u) dx dy.
\end{equation}
At this point, we introduce the auxiliary term mentioned above:
\begin{equation}\label{eq:40}
 \tilde{C}_u:=K^2 s' t' u^2 \int_0^1 \int_0^1 \sqrt{y} \, e^{a(\theta)(2xy-y-x)u}
 \cos(b(\theta)(y-x)u) dx dy,
\end{equation}
where we note that, compared to the right hand-side of
\eqref{eq:30}, we have only replaced $\sqrt{xy}$ by $\sqrt{y}$. For
the moment, assume that $\lim_{u\to\infty} (C_u-\tilde{C}_u) =0$.
Let us compute the limit of $\tilde{C}_u$, recalling that this term
has been defined in \eqref{eq:40}. As in the analysis of the term
$I_1^\ep$, we use the formula $\cos(\theta)=\frac12
(e^{i\theta}+e^{-i\theta})$, so we split $\tilde{C}_u$ as the sum of
two terms (multiplied by $\frac12$), one of which is given by
\[
 K^2 s' t'u^2 \int_0^1 \int_0^1 \sqrt{y} \, e^{u (2xy a(\theta)-y(a(\theta)-ib(\theta))-x(a(\theta)+ib(\theta))} dx dy,
\]
and the other one is the same with $a(\theta)+ib(\theta)$ instead of
$a(\theta)-ib(\theta)$. Integrating first with respect to $x$,
applying l'H\^opital's rule, and doing a change of variables, one
gets that the limit of the above term equals to
\[
 \lim_{u\to\infty}  K^2 s' t' \frac{\int_0^u \sqrt{y} \, e^{(a(\theta)+ib(\theta))y}dy}
 {\left((a(\theta)+ib(\theta))-\frac{(a(\theta)+ib(\theta))^2}{2a(\theta)}\right)\sqrt{u} \, e^{(a(\theta)+ib(\theta))u}}.
\]
It is straightforward to check that the latter limit is $\frac{K^2
s' t'}{a(\theta)^2+b(\theta)^2}$. The limit of the term involving
$a(\theta)+ib(\theta)$ will also be given by $\frac{K^2 s'
t'}{a(\theta)^2+b(\theta)^2}$. Therefore, we have that
\[
 \lim_{\ep\to 0} I_2^\ep = \lim_{u\to\infty} \tilde{C}_u=s' t' K^2 \frac{1}{a(\theta)^2+b(\theta)^2}.
\]
In conclusion, owing to \eqref{eq:51} and the expression of $K$
given in \eqref{eq:50}, the lemma's statement holds in the case
$s=t=0$.

In order to conclude the present step, we need to check that
$\lim_{u\to\infty} (C_u-\tilde{C}_u) =0$, that is
\[
 \lim_{u\to\infty}
 u^2 \int_0^1 \int_0^1 (\sqrt{xy}-\sqrt{y}) \, e^{a(\theta)(2xy-y-x)u}
 \cos(b(\theta)(2xy-y-x)u) dx dy =0.
\]
Let us introduce the notation
\[
D_u:=u^2 \int_0^1 \int_0^1 (\sqrt{xy}-\sqrt{y}) \,
e^{a(\theta)(2xy-y-x)u}dx dy.
\]
Then, it clearly holds that
\[
-D_u \leq C_u - \tilde{C}_u \leq D_u.
\]
In order to apply a sandwich type argument, we will prove that both
$-D_u$ and $D_u$ converge to zero as $u$ tends to infinity. We will
only tackle the term $D_u$, since the analysis of $-D_u$ is
analogous. Note that $D_u=D_u^1-D_u^2$, where
\[
D_u^1= u^2 \int_0^1 \int_0^1 \sqrt{xy} \, e^{a(\theta)(2xy-y-x)u}dx
dy \quad \text{and}\quad D_u^2= u^2 \int_0^1 \int_0^1 \sqrt{y} \,
e^{a(\theta)(2xy-y-x)u}dx dy.
\]
Regarding $D_u^2$, observe that the integral in $x$ can be computed
explicitly and we can argue as follows:
\begin{align*}
\lim_{u\to \infty} D_u^2 & = \lim_{u\to \infty} u \int_0^1
\frac{\sqrt{y}}{a(\theta)(2y-1)}\left(
e^{a(\theta)(y-1)u}-e^{-a(\theta)yu}\right) dy\\
& = \lim_{u\to \infty} \frac{1}{a(\theta)} \frac{\int_0^1
    \frac{\sqrt{y}}{(2y-1)}\left(
    e^{a(\theta)\left(y-\frac{1}{2}\right)u}-e^{-a(\theta)\left(y-\frac{1}{2}\right)u}\right)
    dy}{\frac{e^{u\frac{a(\theta)}{2}}}{u}}\\
& = \lim_{u\to \infty}  \frac{\int_0^1
    \sqrt{y}\left(
    e^{a(\theta)\left(y-\frac{1}{2}\right)u}+e^{-a(\theta)\left(y-\frac{1}{2}\right)u}\right) dy}{\frac{a(\theta)e^{u\frac{a(\theta)}{2}}}{u}-\frac{2e^{u\frac{a(\theta)}{2}}}{u^2}}.
\end{align*}
In the last equality, we have applied l'H\^opital's rule. By
performing a change of variables, the latter expression equals to
\[
\lim_{u\to\infty} \left\{ \frac{\int_0^u\sqrt{y}e^{ y a(\theta)}dy}{\sqrt{u}a(\theta)e^{ua(\theta)}-\frac{2e^{ua(\theta)}}{\sqrt{u}}}+\frac{\int_0^u\sqrt{y}e^{-y a(\theta)}dy}{\sqrt{u}-\frac{1}{\sqrt{u}}}
\right\}.
\]
The second term in the above sum clearly converges to zero as
$u\to\infty$, while the limit of the first one equals to, thanks to
l'H\^opital's rule,
\[
\lim_{u\to\infty} \frac{1}{a(\theta)}\frac{\sqrt{u}\,
e^{a(\theta)u}}{e^{a(\theta)u}(a(\theta)\sqrt{u}+o(\sqrt{u}))}=
\frac{1}{a(\theta)^2}.
\]
Thus, we have proved that $\lim_{u\to \infty}
D_u^2=\frac{1}{a(\theta)^2}$. On the other hand, in order to deal
with $D_u^1$ we will use again a sandwich type argument, as follows.
First, note that we trivially have $D_u^1\leq D_u^2$. Next, applying
the changes of variables $v=uy$ and $z={xv}$, we end up with
\begin{align*}
\lim_{u\to\infty}D_u^1 & = \lim_{u\to\infty} \sqrt{u} \int_0^u
\int_0^v \frac{\sqrt{z}}{v} \, e^{-a(\theta)(v+\frac{zu}{v}-2z)}dz
dv\\
& \geq \lim_{u\to\infty} \sqrt{u} \int_0^u \int_0^v
\frac{\sqrt{z}}{v} \, e^{-a(\theta)(u-z)}dz dv.
\end{align*}
Observe that the latter limit equals to $\frac{1}{a(\theta)^2}$
because it corresponds to the limit of $B_u$ defined above in the
particular case of $s'=t'=K=1$ and $b=0$. Hence, we obtain that
\[
\lim_{u\to \infty} D_u^1=\frac{1}{a(\theta)^2}
\]
and therefore $\lim_{u\to\infty} D_u =0$.

\smallskip

{\it{Step 3.}} Assume that either $s\neq 0$ or $t\neq 0$. By step 1, recall that we have
\[
 \E\big[ |\Delta_{s,t} X_\ep(s',t')|^2\big]= 2(I^\ep_1+I^\ep_2),
\]
where the terms on the right hand-side have been defined in \eqref{eq:52} and \eqref{eq:23}, respectively.
Set
\[
 F^\ep(s,t):=\ep^2 K^2 \int_0^{\frac t\ep}\int_0^{\frac s\ep}\int_0^{\frac t\ep}\int_0^{\frac s\ep}
 f(x_1,x_2,y_1,y_2) 1_{\{x_1\leq x_2, y_1\leq y_2\}} dx_1 dy_1 dx_2 dy_2,
\]
where $f(x_1,x_2,y_1,y_2):= \sqrt{x_1 x_2 y_1 y_2}\,
e^{-a(\theta)(x_2 y_2-x_1 y_1)} \cos(b(\theta)(x_2 y_2-x_1 y_1))$,
and
\[
 G^\ep(s,t):=\ep^2 K^2 \int_0^{\frac t\ep}\int_0^{\frac s\ep}\int_0^{\frac t\ep}\int_0^{\frac s\ep}
 g(x_1,x_2,y_1,y_2) 1_{\{x_1\leq x_2, y_1\leq y_2\}} dx_1 dy_1 dx_2 dy_2,
\]
where $g(x_1,x_2,y_1,y_2):= \sqrt{x_1 x_2 y_1 y_2}\,
e^{-a(\theta)((y_2-y_1)x_1 + (x_2-x_1)y_1)}
\cos(b(\theta)((y_2-y_1)x_1 - (x_2-x_1)y_1))$. Observe that
$I_1^\ep$ and $I_2^\ep$ can be written as follows:
\begin{align*}
 I_1^\ep & = \Delta_{s,t}F^\ep(s',t') -
 \ep^2 K^2 \int_{\frac t\ep}^{\frac {t'}\ep}\int_{\frac s\ep}^{\frac {s'}\ep}\int_{\frac t\ep}^{\frac {t'}\ep}
 \int_0^{\frac s\ep}
 f(x_1,x_2,y_1,y_2) 1_{\{y_1\leq y_2\}} dx_1 dy_1 dx_2 dy_2 \\
 & \qquad \quad -
 \ep^2 K^2 \int_{\frac t\ep}^{\frac {t'}\ep}\int_{\frac s\ep}^{\frac {s'}\ep}\int_{0}^{\frac {{\color{red}{t}}}\ep}
 \int_0^{\frac s\ep}
 f(x_1,x_2,y_1,y_2) dx_1 dy_1 dx_2 dy_2 \\
 & \qquad \quad -
 \ep^2 K^2 \int_{\frac t\ep}^{\frac {t'}\ep}\int_{\frac s\ep}^{\frac {s'}\ep}\int_{0}^{\frac t\ep}
 \int_{\frac s\ep}^{\frac {s'}\ep}
 f(x_1,x_2,y_1,y_2) 1_{\{x_1\leq x_2\}} dx_1 dy_1 dx_2 dy_2\\
 & =: \Delta_{s,t}F^\ep(s',t') - I_{11}^\ep - I_{12}^\ep - I_{13}^\ep,
\end{align*}
and
\[
 I_2^\ep = \Delta_{s,t} G^\ep(s',t') - I_{21}^\ep - I_{22}^\ep - I_{23}^\ep,
\]
where $I_{2i}^\ep$, $i=1,2,3$, are defined analogously by using the
function $g$. By step 2, one verifies that
\[
\lim_{\ep\to 0} \Delta_{s,t}F^\ep(s',t') = \lim_{\ep\to 0}
\Delta_{s,t}G^\ep(s',t') = \frac12 (s'-s)(t'-t).
\]
In order to conclude the proof, it suffices to check that
$I_{ji}^\ep$ converges to zero as $\ep\to 0$, for all $j=1,2$ and
$i=1,2,3$. For this, we estimate any $I_{ji}^\ep$ by
$\tilde{I}_{ji}^\ep$, where the latter are defined by simply
bounding the cosinus by $1$. Next, we note that
$\tilde{I}_{1,i}^\ep\leq \tilde{I}_{2,i}^\ep$, for all $i=1,2,3$,
and that any of the $\tilde{I}_{2,i}^\ep$ can be bounded by
\begin{equation}\label{eq:56}
 \ep^2 K^2 \int_{0}^{\frac {t'}\ep}\int_{0}^{\frac {s}\ep}\int_{0}^{y_2}
 \int_{\frac s\ep}^{\frac {s'}\ep}
 \sqrt{x_1 x_2 y_1 y_2}\,
 e^{-a(\theta)((y_2-y_1)x_1 + (x_2-x_1)y_1)}  dx_2 dy_1 dx_1 dy_2.
\end{equation}
In this integral, we perform the changes of variables
$\bar{x}_i:=\ep x_i$ and $\bar{y}_i:=\ep y_i$, $i=1,2$, we set
$u:=\frac{1}{\ep^2}$, we use that $\bar{x}_2\leq s'$ and we
integrate with respect to $\bar{x}_2$. Thus, \eqref{eq:56} can be
bounded, up to some positive constant, by (using again the notation
$x_i$ and $y_i$ for the variables)
\[
 u \int_0^{t'} \int_0^s \int_0^{y_2} \frac{\sqrt{x_1 y_2}}{\sqrt{y_1}}
 \, e^{-a(\theta) ((y_2-y_1)x_1 + (s-x_1)y_1)u} dy_1 dx_1 dy_2.
\]
Estimating now $y_2$ by $t'$ inside the square root and integrating
in $y_2$, the above expression can be bounded by (up to some
constant)
\[
\int_0^s \int_0^{t'} \frac{1}{\sqrt{x_1 y_1}} \,
e^{-a(\theta)u(s-x_1)y_1} dy_1 dx_1.
\]
This expression converges to zero as $u\to\infty$, by the Monotone
convergence theorem.
\end{proof}

\medskip

Here is the second auxiliary result needed to prove Proposition
\ref{prop:6}.

\begin{lemma}\label{lem:99}
Let $(0,0)\leq (s,t)\leq (s',t')\leq (S,T)$. Then, there exists a
sequence $\{C_\ep\}_{\ep>0}$ such that $\lim_{\ep\to 0}
C_\ep=4(s'-s)^2(t'-t)^2$ and
\[
\E\left[ \left( \E\left[ |\Delta_{s,t} X^\ep(s',t')|^2
\,|\mathcal{F}^\ep_{s,T}\right]\right)^2\right]\leq C_\ep.
\]
\end{lemma}

\begin{proof} We split the proof in four steps.

\smallskip

{\it{Step 1.}} By definition of the random field $X^\ep$, we first
observe that
\begin{align*}
& \E\left[ |\Delta_{s,t} X^\ep(s',t')|^2
\,|\mathcal{F}^\ep_{s,T}\right] \\
& \qquad = K^2\ep^2 \int_{\frac t\ep}^{\frac {t'}\ep}\int_{\frac
s\ep}^{\frac {s'}\ep}\int_{\frac t\ep}^{\frac {t'}\ep}\int_{\frac
s\ep}^{\frac {s'}\ep} \sqrt{x_1x_2y_1y_2} \; \E\left[
e^{i\theta\left(L(x_2,y_2)-L(x_1,y_1)\right)}|\mathcal{F}^\ep_{s,T}\right]
\, dx_1dx_2dy_1dy_2.
\end{align*}
In order to compute the above conditional expectation, we have to
consider all possible orders of $x_1,x_2$ and $y_1,y_2$,
respectively, which corresponds to a total of 4 possibilities.
Hence,
\begin{align*}
\E\left[ |\Delta_{s,t} X^\ep(s',t')|^2
\,|\mathcal{F}^\ep_{s,T}\right] &  = K^2\ep^2 \int_{\frac
t\ep}^{\frac {t'}\ep}\int_{\frac s\ep}^{\frac {s'}\ep}\int_{\frac
t\ep}^{y_2}\int_{\frac s\ep}^{x_2} \sqrt{x_1x_2y_1y_2}  \,
e^{i\theta\left(L(\frac s\ep,y_2)-L(\frac
s\ep,y_1)\right)} \\
& \qquad \qquad \times
e^{-\Psi(\theta)\left((x_2-x_1)y_2+(y_2-y_1)(x_1-\frac s\ep)\right)}
dx_1dy_1dx_2dy_2\\
& \;\; + K^2\ep^2 \int_{\frac t\ep}^{\frac {t'}\ep}\int_{\frac
s\ep}^{\frac {s'}\ep}\int_{\frac t\ep}^{y_2}\int_{\frac s\ep}^{x_2}
\sqrt{x_1x_2y_1y_2} \, e^{-i\theta\left(L(\frac s\ep,y_2)-L(\frac
s\ep,y_1)\right)} \\
& \qquad \qquad \times  e^{-\Psi(\theta)(x_2-x_1)y_1}
e^{-\Psi(-\theta)(y_2-y_1)(x_1-\frac s\ep)}
dx_1dy_1dx_2dy_2\\
& \;\; + K^2\ep^2 \int_{\frac t\ep}^{\frac {t'}\ep}\int_{\frac
s\ep}^{\frac {s'}\ep}\int_{\frac t\ep}^{y_2}\int_{\frac s\ep}^{x_2}
\sqrt{x_1x_2y_1y_2} \, e^{i\theta\left(L(\frac s\ep,y_2)-L(\frac
s\ep,y_1)\right)} \\
& \qquad \qquad \times e^{-\Psi(-\theta)(x_2-x_1)y_1}
e^{-\Psi(\theta)(y_2-y_1)(x_1-\frac s\ep)}
dx_1dy_1dx_2dy_2\\
& \;\; + K^2\ep^2 \int_{\frac t\ep}^{\frac {t'}\ep}\int_{\frac
s\ep}^{\frac {s'}\ep}\int_{\frac t\ep}^{y_2}\int_{\frac s\ep}^{x_2}
\sqrt{x_1x_2y_1y_2} \, e^{-i\theta\left(L(\frac s\ep,y_2)-L(\frac
s\ep,y_1)\right)} \\
& \qquad \qquad \times
e^{-\Psi(-\theta)\left((x_2-x_1)y_2+(y_2-y_1)(x_1-\frac
s\ep)\right)} dx_1dy_1dx_2dy_2.
\end{align*}
We have also applied changes of variables in order to have $x_1\leq
x_2$ and $y_1\leq y_2$ in all terms. We denote by $A_i^\ep$,
$i=1,2,3,4$, the above four terms, respectively. Thus, we have
\[
\E\left[ \left( \E\left[ |\Delta_{s,t} X^\ep(s',t')|^2
\,|\mathcal{F}^\ep_{s,T}\right]\right)^2\right]= \sum_{i,j=1}^4 \E
\left[A_i^\ep A_j^\ep\right].
\]
For the sake of clarity, we will only analyze one of the terms in
the above sum, since the other ones can be treated exactly in the
same way. So, we proceed to tackle the term
$\E\left[(A_1^\ep)^2\right]$. In fact, by Fubini theorem, we have
that
\begin{align}
\E\left[(A_1^\ep)^2\right] & = K^4\ep^4
\int_{\frac{t}{\ep}}^{\frac{t'}{\ep}}\int_{\frac{s}{\ep}}^{\frac{s'}{\ep}}
\int_{\frac{t}{\ep}}^{y_4}\int_{\frac{s}{\ep}}^{x_4}\int_{\frac{t}{\ep}}^{\frac{t'}{\ep}}
\int_{\frac{s}{\ep}}^{\frac{s'}{\ep}}\int_{\frac{t}{\ep}}^{y_2}\int_{\frac{s}{\ep}}^{x_2}
\sqrt{x_1x_2x_3x_4}\sqrt{y_1y_2y_3y_4} \nonumber
\\
& \qquad \times \mathbb{E}\left[e^{i\theta\left(
L(\frac{s}{\ep},y_2)- L(\frac{s}{\ep},y_1)+L(\frac{s}{\ep},y_4)-
L(\frac{s}{\ep},y_3)\right)} \right] \nonumber \\
& \qquad  \times e^{-\Psi(\theta)\left(
(x_2-x_1)y_2+(y_2-y_1)(x_1-\frac{s}{\ep})+
(x_4-x_3)y_4+(y_4-y_3)(x_3-\frac{s}{\ep})\right)}\nonumber \\
& \qquad \times dx_1dy_1dx_2dy_2dx_3dy_3dx_4dy_4. \label{eq:555}
\end{align}
Note that in the above integral we have $y_1\leq y_2$ and $y_3\leq
y_4$. However, in order to compute the expectation in \eqref{eq:555},
we need to consider all possible orders of the variables
$y_1,y_2,y_3,y_4$, with the restrictions $y_1\leq y_2$ and $y_3\leq
y_4$. This amounts to take into account 6 different possibilities,
which we split in two groups:
\begin{itemize}
\item[(i)] $y_1\leq y_2 \leq y_3 \leq y_4$ and $y_3\leq y_4 \leq y_1 \leq
y_2$,
\item[(ii)] $y_1\leq y_3 \leq y_2 \leq y_4$, $y_1\leq y_3 \leq y_4 \leq
y_2$, $y_3\leq y_1 \leq y_4 \leq y_2$ and $y_3\leq y_1 \leq y_2 \leq
y_4$.
\end{itemize}
Then, we have that
\begin{equation}\label{eq:600}
 \E\left[(A_1^\ep)^2\right]=\sum_{k=1}^6 B_k^\ep(1,1),
\end{equation}
where $B_1^\ep(1,1), B_2^\ep(1,1)$ correspond to \eqref{eq:555} with the orders of (i), respectively,
while $B_k^\ep(1,1)$, $k=3,4,5,6$, correspond to
\eqref{eq:555} with the orders of (ii), respectively. It turns out that we have a similar decomposition of
any of the terms $\E
\big[A_i^\ep A_j^\ep\big]$, which we denote by
\[
 \E
\left[A_i^\ep A_j^\ep\right]=\sum_{k=1}^6 B_k^\ep(i,j).
\]
Hence
\begin{equation}\label{eq:554}
\E\left[ \left( \E\left[ |\Delta_{s,t} X^\ep(s',t')|^2
\,|\mathcal{F}^\ep_{s,T}\right]\right)^2\right]= \sum_{i,j=1}^4 \sum_{k=1}^6 B_k^\ep(i,j).
\end{equation}
In the next two steps, we will focus on the analysis of (some of)
the terms in the decomposition \eqref{eq:600} of
$\E\left[(A_1^\ep)^2\right]$. As already mentioned, the terms
arising from $\E \big[A_i^\ep A_j^\ep\big]$ can be treated
analogously. We will come back to expansion \eqref{eq:554} later in
step 4.

\smallskip

{\it{Step 2.}} We claim that, for any $k=3,4,5,6$, it holds
\begin{align}
|B_k^\ep(1,1)| & \leq  K^4 \ep^4 \int_D \sqrt{x_1x_2x_3x_4}\sqrt{y_1y_2y_3y_4} \,
1_{\{x_1\leq x_2\}} 1_{\{x_3\leq x_4\}} 1_{\{y_1\leq y_2\leq y_3\leq
y_4\}}\nonumber \\
& \qquad \qquad \times e^{-a(\theta)\left((x_4-x_3)\frac t\ep +
(x_2-x_1)\frac t\ep + (y_4-y_3)\frac s\ep + (y_2-y_1)\frac s\ep +
(y_3-y_2)(x_1-\frac s\ep)\right)}\nonumber \\
& \qquad \qquad \times dx_1dy_1dx_2dy_2dx_3dy_3dx_4dy_4,
\label{eq:556}
\end{align}
where $D:=[\frac s\ep,\frac{s'}{\ep}]^4\times [\frac
t\ep,\frac{t'}{\ep}]^4$, and we recall that $a(\theta)$ is the real
part of $\Psi(\theta)$. We prove this estimate for $B_3^\ep(1,1)$.
For the remaining terms the argument is completely analogous. So,
let us assume that in \eqref{eq:555} we have the order $y_1\leq y_3
\leq y_2 \leq y_4$. In this case, the expectation in \eqref{eq:555}
equals to
\[
 e^{-\Psi(\theta)\left((y_4-y_2)\frac s\ep + (y_3-y_1)\frac s\ep + 2(y_2-y_3)\frac s\ep\right)}.
\]
Plugging this term in \eqref{eq:555} and shifting the modulus inside the integral, we can infer that
\begin{align*}
 |B_3^\ep(1,1)| & \leq
 K^4\ep^4
\int_D
\sqrt{x_1x_2x_3x_4}\sqrt{y_1y_2y_3y_4} 1_{\{x_1\leq x_2\}} 1_{\{x_3\leq x_4\}} 1_{\{y_1\leq y_3\leq y_2\leq
y_4\}} \nonumber \\
& \qquad \times e^{-a(\theta)\left((x_4-x_3)\frac t\ep +
(x_2-x_1)\frac t\ep + (y_4-y_2)\frac s\ep + (y_3-y_1)\frac s\ep +
(y_2-y_1)(x_1-\frac s\ep)\right)} \nonumber \\
& \qquad \times  e^{-a(\theta)\left( 2(y_2-y_3)\frac s\ep +
(y_4-y_3)(x_3-\frac s\ep)\right)}
dx_1dy_1dx_2dy_2dx_3dy_3dx_4dy_4\\
& \leq K^4\ep^4
\int_D
\sqrt{x_1x_2x_3x_4}\sqrt{y_1y_2y_3y_4} 1_{\{x_1\leq x_2\}} 1_{\{x_3\leq x_4\}} 1_{\{y_1\leq y_3\leq y_2\leq
y_4\}} \nonumber \\
& \qquad \times e^{-a(\theta)\left((x_4-x_3)\frac t\ep +
(x_2-x_1)\frac t\ep + (y_4-y_2)\frac s\ep + (y_3-y_1)\frac s\ep +
(y_2-y_1)(x_1-\frac s\ep)\right)} \\
& \qquad \times
dx_1dy_1dx_2dy_2dx_3dy_3dx_4dy_4.
\end{align*}
Performing a change of variable, we obtain that the latter term equals to
\begin{align*}
&K^4\ep^4
\int_D
\sqrt{x_1x_2x_3x_4}\sqrt{y_1y_2y_3y_4} 1_{\{x_1\leq x_2\}} 1_{\{x_3\leq x_4\}} 1_{\{y_1\leq y_2\leq y_3\leq
y_4\}} \nonumber \\
& \qquad \times e^{-a(\theta)\left((x_4-x_3)\frac t\ep +
(x_2-x_1)\frac t\ep + (y_4-y_3)\frac s\ep + (y_2-y_1)\frac s\ep +
(y_3-y_1)(x_1-\frac s\ep)\right)}\\
& \qquad \times
dx_1dy_1dx_2dy_2dx_3dy_3dx_4dy_4.
\end{align*}
In order to obtain \eqref{eq:556}, it suffices to observe that, in
the domain of integration, it holds that $(y_3-y_1)(x_1-\frac
s\ep)\geq (y_3-y_2)(x_1-\frac s\ep)$.

\smallskip

{\it{Step 3.}} Here, we prove that the right hand-side of \eqref{eq:556} converges to zero as $\ep\to 0$. Let us introduce the following notation:
\begin{align*}
\beta^\ep & := \ep^4 \int_D \sqrt{x_1x_2x_3x_4}\sqrt{y_1y_2y_3y_4} \,
1_{\{x_1\leq x_2\}} 1_{\{x_3\leq x_4\}} 1_{\{y_1\leq y_2\leq y_3\leq
y_4\}}\nonumber \\
& \qquad \qquad \times e^{-a(\theta)\left((x_4-x_3)\frac t\ep +
(x_2-x_1)\frac t\ep + (y_4-y_3)\frac s\ep + (y_2-y_1)\frac s\ep +
(y_3-y_2)(x_1-\frac s\ep)\right)}\nonumber \\
& \qquad \qquad \times dx_1dy_1dx_2dy_2dx_3dy_3dx_4dy_4,
\end{align*}
so we want to check that $\lim_{\ep\to0} \beta^\ep =0$.

To start with, in the expression of $\beta^\ep$ we bound the two
square roots by using the upper limit of any $x_i$ and $y_i$. Next,
we integrate with respect to $x_4$, $x_3$ and $x_2$. We also use the
fact that, according to the statement of Proposition \ref{prop:6},
we may assume that $t>0$. Thus,
\begin{align*}
 \beta^\ep & \leq C \ep \int_{\frac s\ep}^{\frac{s'}{\ep}} \int_{\frac t\ep}^{\frac{t'}{\ep}} \int_{\frac t\ep}^{y_3}\int_{y_3}^{\frac{t'}{\ep}}
 \int_{\frac t\ep}^{y_2}
 e^{-a(\theta)\left((y_4-y_3)\frac s\ep + (y_2-y_1)\frac s\ep +
(y_3-y_2)(x_1-\frac s\ep)\right)} \, dy_1 dy_4 dy_2 dy_3 dx_1.
\end{align*}
At this point, we integrate with respect to $y_1$ and $y_4$, thus
\begin{align*}
 \beta^\ep & \leq C \ep^3
 \int_{\frac s\ep}^{\frac{s'}{\ep}} \int_{\frac t\ep}^{\frac{t'}{\ep}} \int_{\frac t\ep}^{y_3}
 e^{-a(\theta)(y_3-y_2)(x_1-\frac s\ep)} dy_2 dy_3 dx_1\\
 & = C \ep^3 \int_\ep^{\frac{s'-s}{\ep}}  \int_{\frac t\ep}^{\frac{t'}{\ep}} \int_{\frac t\ep}^{y_3}
 e^{-a(\theta)(y_3-y_2)x} dy_2 dy_3 dx\\
 & \qquad + C \ep^3 \int_0^\ep  \int_{\frac t\ep}^{\frac{t'}{\ep}} \int_{\frac t\ep}^{y_3}
 e^{-a(\theta)(y_3-y_2)x} dy_2 dy_3 dx.
\end{align*}
Note that the second term in the latter sum may be bounded, up to some positive constant, by $\ep^2$, which converges to zero.
Regarding the first term, it can be bounded by
\[
 C \ep^2 \int_\ep^{\frac{s'-s}{\ep}}  \frac1x \, dx = C \ep^2 \left(\ln(s'-s)-2\ln(\ep)\right),
\]
which also converges to zero as $\ep\to0$.

\smallskip

{\it{Step 4.}} By \eqref{eq:554} in step 1 and steps 2 and 3, we have that
\begin{align}
 \E\left[ \left( \E\left[ |\Delta_{s,t} X^\ep(s',t')|^2
\,|\mathcal{F}^\ep_{s,T}\right]\right)^2\right] & = \sum_{i,j=1}^4 \sum_{k=1}^6 B_k^\ep(i,j)\nonumber \\
& = \sum_{i,j=1}^4 \sum_{k=1}^2 B_k^\ep(i,j) + \rho_\ep,
\label{eq:557}
\end{align}
where we recall that $B_1^\ep(i,j)$ and $B_2^\ep(i,j)$ are the terms in
the decomposition of $\E \big[A_i^\ep A_j^\ep\big]$ with the orders of (i), respectively,
and $\lim_{\ep\to 0}\rho_\ep=0$.

Focusing again (only) on the case $i=j=1$, one easily verifies that
\begin{align*}
 \sum_{k=1}^2 B_k^\ep(1,1) & = K^4 \ep^2 \int_D \sqrt{x_1x_2x_3x_4}\sqrt{y_1y_2y_3y_4} \,
 1_{\{x_1\leq x_2\}} 1_{\{x_3\leq x_4\}}
 1_{\left\{ \{y_1\leq y_2\leq y_3\leq y_4\}\cup \{y_3\leq y_4\leq y_1\leq y_2\}\right\}}\nonumber \\
 & \qquad \times e^{-\Psi(\theta)\left((x_2-x_1)y_2+(x_4-x_3)y_4 + (y_2-y_1)x_1 + (y_4-y_3)x_3\right)}\\
 & \qquad \times dx_1dy_1dx_2dy_2dx_3dy_3dx_4dy_4,
\end{align*}
where we recall that $D:=[\frac s\ep,\frac{s'}{\ep}]^4\times [\frac
t\ep,\frac{t'}{\ep}]^4$. Observing that
\begin{align*}
 1_{\left\{ \{y_1\leq y_2\leq y_3\leq y_4\}\cup \{y_3\leq y_4\leq y_1\leq y_2\}\right\}}\leq
 1_{\{y_1\leq y_2\}} 1_{\{y_3\leq y_4\}},
\end{align*}
we end up with
\begin{align*}
 \sum_{k=1}^2 B_k^\ep(1,1) & \leq K^4 \ep^2 \int_D \sqrt{x_1x_2x_3x_4}\sqrt{y_1y_2y_3y_4} \,
 1_{\{x_1\leq x_2\}} 1_{\{x_3\leq x_4\}}
 1_{\{y_1\leq y_2\}} 1_{\{y_3\leq y_4\}}\nonumber \\
 & \qquad \times e^{-\Psi(\theta)\left((x_2-x_1)y_2+(x_4-x_3)y_4 + (y_2-y_1)x_1 + (y_4-y_3)x_3\right)}\\
 & \qquad \times dx_1dy_1dx_2dy_2dx_3dy_3dx_4dy_4.
\end{align*}
One can get similar estimates for $B_1^\ep(i,j)+B_2^\ep(i,j)$ with $i,j\neq 1$.
Gathering all the resulting bounds together, it can be
verified that
\begin{equation}\label{eq:586}
\E\left[ \left( \E\left[ |\Delta_{s,t} X^\ep(s',t')|^2
\,|\mathcal{F}^\ep_{s,T}\right]\right)^2\right] \leq \Theta_\ep^2 + {\rho}_\ep,
\end{equation}
where
\begin{align*}
 \Theta_\ep & = K^2\ep^2 \int_{\frac t\ep}^{\frac {t'}\ep} \int_{\frac t\ep}^{y_2}
 \int_{\frac s\ep}^{\frac {s'}\ep} \int_{\frac s\ep}^{x_2}
\sqrt{x_1x_2y_1y_2}  \,
e^{-\Psi(\theta)\left((x_2-x_1)y_2+(y_2-y_1)x_1\right)}
dx_1dx_2dy_1dy_2\\
& \quad +  K^2\ep^2 \int_{\frac t\ep}^{\frac {t'}\ep} \int_{\frac t\ep}^{y_2}
 \int_{\frac s\ep}^{\frac {s'}\ep} \int_{\frac s\ep}^{x_2}
\sqrt{x_1x_2y_1y_2}  \, e^{-\Psi(\theta)(x_2-x_1)y_1}
e^{-\Psi(-\theta)(y_2-y_1)x_1}
dx_1dx_2dy_1dy_2\\
& \quad +  K^2\ep^2 \int_{\frac t\ep}^{\frac {t'}\ep} \int_{\frac t\ep}^{y_2}
 \int_{\frac s\ep}^{\frac {s'}\ep} \int_{\frac s\ep}^{x_2}
\sqrt{x_1x_2y_1y_2}  \, e^{-\Psi(-\theta)(x_2-x_1)y_1}
e^{-\Psi(\theta)(y_2-y_1)x_1}
dx_1dx_2dy_1dy_2\\
& \quad +  K^2\ep^2 \int_{\frac t\ep}^{\frac {t'}\ep} \int_{\frac
t\ep}^{y_2}
 \int_{\frac s\ep}^{\frac {s'}\ep} \int_{\frac s\ep}^{x_2}
\sqrt{x_1x_2y_1y_2}  \,
e^{-\Psi(-\theta)\left((x_2-x_1)y_2+(y_2-y_1)x_1\right)}
dx_1dx_2dy_1dy_2.
\end{align*}
Note that $\Theta_\ep$ coincides with the right hand-side of equality \eqref{eq:585} in the proof of Lemma
\ref{lem:25}, where in the latter it was precisely proved that
\[
 \lim_{\ep\to0} \Theta_\ep= 2(t'-t)(s'-s).
\]
Therefore, by \eqref{eq:586} and recalling that $\lim_{\ep\to0}\rho_\ep=0$, we conclude the proof by taking
$C_\ep:=\Theta_\ep^2+\rho_\ep$.
\end{proof}

\medskip

We can now provide the proof of Proposition \ref{prop:6}.

{\it{Proof of Proposition \ref{prop:6}.}} We prove that, for all $0\leq s_1<\cdots<s_n\leq s$ and
$0\leq t_1<\cdots<t_n\leq T$, and any continuous and bounded function
$\varphi: \bC\to\R$, we have
\[
 \E_{\P}\left[\varphi(X(s_1,t_1),\dots,X(s_n,t_n))\left(\big(\Delta_{s,t} \mathsf{Re}(X)(s',t')\big)^2
 -(s'-s)(t'-t)\right)\right]=0
\]
and
\[
 \E_{\P}\left[\varphi(X(s_1,t_1),\dots,X(s_n,t_n))\left(\big(\Delta_{s,t} \mathsf{Im}(X)(s',t')\big)^2
 -(s'-s)(t'-t)\right)\right]=0.
\]
Since $\P_\ep$ converges to $\P$ weakly in $\C([0,S]\times [0,T];\bC)$, it suffices to check that
\begin{equation}\label{eq:200}
 \lim_{\ep\to0} A_\ep=0 \quad \text{and}\quad \lim_{\ep\to0} B_\ep=0,
\end{equation}
where
\[
 A_\ep:=  \E\left[\varphi(X_\ep(s_1,t_1),\dots,X_\ep(s_n,t_n))
 \left(\big(\Delta_{s,t} \mathsf{Re}(X_\ep)(s',t')\big)^2
 -(s'-s)(t'-t)\right)\right]
\]
and
\[
 B_\ep:= \E\left[\varphi(X_\ep(s_1,t_1),\dots,X_\ep(s_n,t_n))\left(\big(\Delta_{s,t}
 \mathsf{Im}(X_\ep)(s',t')\big)^2
 -(s'-s)(t'-t)\right)\right].
\]
Indeed, in order to check the validity of the limits in
\eqref{eq:200}, we will prove that
\[
 \lim_{\ep\to0} (A_\ep+B_\ep)=0 \quad \text{and}\quad \lim_{\ep\to0}
 (A_\ep-B_\ep)=0.
\]
We will first deal with the limit of $A_\ep+B_\ep$. More precisely,
we have that
\begin{align*}
A_\ep+B_\ep & = \E\left[\varphi(X_\ep(s_1,t_1),\dots,X_\ep(s_n,t_n))
 \left( |\Delta_{s,t}X^\ep(s',t')|^2 - 2(s-s')(t-t')\right)\right]\\
 & = \E\left[\varphi(X_\ep(s_1,t_1),\dots,X_\ep(s_n,t_n))
 \left( \E\left[ |\Delta_{s,t}X^\ep(s',t')|^2\, | \mathcal{F}^\ep_{s,T}\right]
 - 2(s-s')(t-t')\right)\right].
\end{align*}
Hence, to prove that $\lim_{\ep\to0}(A_\ep+B_\ep)=0$, it is enough
to check that $\E\big[ |\Delta_{s,t}X^\ep(s',t')|^2\, |
\mathcal{F}^\ep_{s,T}\big]$ converges in $L^2(\Omega)$ to
$2(s-s')(t-t')$, as $\ep\to0$. Indeed, by Lemma \ref{lem:99}, we
have:
\begin{align}
& \E\left[\left(\E\big[ |\Delta_{s,t}X^\ep(s',t')|^2\, |
\mathcal{F}^\ep_{s,T}\big] - 2(s-s')(t-t')\right)^2\right] \nonumber\\
& \qquad \leq C_\ep -4(s-s')(t-t') \E\big[
|\Delta_{s,t}X^\ep(s',t')|^2 \big] + 4(s-s')^2(t-t')^2,
\label{eq:201}
\end{align}
where $\lim_{\ep\to0}C_\ep=4(s'-s)^2(t'-t)^2$. So, by Lemma
\ref{lem:25}, the right hand-side of \eqref{eq:201} converges to
zero as $\ep\to0$.

Let us now deal with the limit of $A_\ep-B_\ep$. To start with, note
that
\begin{align*}
 A_\ep-B_\ep &  = \frac12\, \E\Bigg[
\varphi(X_\ep(s_1,t_1),\dots,X_\ep(s_n,t_n)) \\
& \quad \times \left.\left\{ \left( K \ep \int_{\frac
t\ep}^{\frac{t'}{\ep}} \int_{\frac s\ep}^{\frac{s'}{\ep}}
\sqrt{xy}\, e^{i\theta L(x,y)} dx dy\right)^2 + \left( K \ep
\int_{\frac t\ep}^{\frac{t'}{\ep}} \int_{\frac
s\ep}^{\frac{s'}{\ep}} \sqrt{xy}\, e^{-i\theta L(x,y)} dx
dy\right)^2\right\}\right].
\end{align*}
We are going to prove that $\lim_{\ep\to0} \Lambda_\ep=0$, where
\[
\Lambda_\ep:= \E\left[ \varphi(X_\ep(s_1,t_1),\dots,X_\ep(s_n,t_n))
\left( K \ep \int_{\frac t\ep}^{\frac{t'}{\ep}} \int_{\frac
s\ep}^{\frac{s'}{\ep}} \sqrt{xy}\, e^{i\theta L(x,y)} dx
dy\right)^2\right].
\]
The limit involving the complex conjugate $e^{-i\theta L(x,y)}$ can
be tackled using analogous arguments. Expanding the squared integral
of $\Lambda_\ep$, we end up with
\begin{align*}
\Lambda_\ep & = \E\Bigg[
\varphi(X_\ep(s_1,t_1),\dots,X_\ep(s_n,t_n))\\
& \quad \qquad \times \left.  K^2 \ep^2 \int_{\frac
t\ep}^{\frac{t'}{\ep}} \int_{\frac s\ep}^{\frac{s'}{\ep}}\int_{\frac
t\ep}^{\frac{t'}{\ep}} \int_{\frac s\ep}^{\frac{s'}{\ep}}
\sqrt{x_1x_2y_1y_2} \, e^{i\theta(L(x_1,y_1)+L(x_2,y_2))}
dx_1dx_2dy_1dy_2\right].
\end{align*}
As we have already done several times throughout the paper, we
consider the four possible orders of $x_1,x_2$ and $y_1,y_2$ and, in
each of these terms, we apply a change of variables so that we have
$x_1\leq x_2$ and $y_1\leq y_2$. Thus,
\begin{align*}
\Lambda_\ep & = \E\Bigg[
\varphi(X_\ep(s_1,t_1),\dots,X_\ep(s_n,t_n))\\
& \quad \qquad \times \left( 2 K^2 \ep^2 \int_{\frac
t\ep}^{\frac{t'}{\ep}} \int_{\frac s\ep}^{\frac{s'}{\ep}}\int_{\frac
t\ep}^{y_2} \int_{\frac s\ep}^{x_2} \sqrt{x_1x_2y_1y_2} \,
e^{i\theta(L(x_1,y_1)+L(x_2,y_2))} dx_1dy_1dx_2dy_2\right. \\
& \qquad \qquad + \left.\left. 2 K^2 \ep^2 \int_{\frac
t\ep}^{\frac{t'}{\ep}} \int_{\frac s\ep}^{\frac{s'}{\ep}}\int_{\frac
t\ep}^{y_2} \int_{\frac s\ep}^{x_2} \sqrt{x_1x_2y_1y_2} \,
e^{i\theta(L(x_2,y_1)+L(x_1,y_2))} dx_1dy_1dx_2dy_2\right) \right].
\end{align*}
At this point, the idea is to write $L(x_1,y_1)+L(x_2,y_2)$ and
$L(x_2,y_1)+L(x_1,y_2)$ as sums of suitable rectangular increments
of $L$ (which will be clearly specified in the next equation), and
use the property of independent (rectangular) increments of $L$ (see
Definition\ref{def:1}). Proceeding in this way, one obtains that
\begin{align*}
\Lambda_\ep & = \E\left[
\varphi(X_\ep(s_1,t_1),\dots,X_\ep(s_n,t_n))\, e^{i2\theta L\left(\frac s\ep,\frac t\ep\right)}\right] \\
& \qquad \times \left( 2 K^2 \ep^2 \int_{\frac
t\ep}^{\frac{t'}{\ep}} \int_{\frac s\ep}^{\frac{s'}{\ep}}\int_{\frac
t\ep}^{y_2} \int_{\frac s\ep}^{x_2} \sqrt{x_1x_2y_1y_2} \,\,
\E\left[e^{i\theta\left(\Delta_{x_1,y_1}L(x_2,y_2) +
\Delta_{0,y_1}L(x_1,y_2) + \Delta_{x_1,0}L(x_2,y_1)\right)}\right]
 \right. \\
& \qquad \qquad \qquad \times \E\left[e^{i2\theta\left(\Delta_{\frac
s\ep,\frac t\ep}L(x_1,y_1) + \Delta_{0,\frac t\ep}L(\frac s\ep,y_1)
+ \Delta_{\frac s\ep,0}L(x_1,\frac t\ep)\right)}\right]
dx_1dy_1dx_2dy_2\\
& \qquad \qquad + 2 K^2 \ep^2 \int_{\frac t\ep}^{\frac{t'}{\ep}}
\int_{\frac s\ep}^{\frac{s'}{\ep}}\int_{\frac t\ep}^{y_2}
\int_{\frac s\ep}^{x_2} \sqrt{x_1x_2y_1y_2} \,\,
\E\left[e^{i\theta\left(\Delta_{0,y_1}L(x_1,y_2) +
\Delta_{x_1,0}L(x_2,y_1)\right)}\right]\\
& \qquad \qquad \qquad \times \E\left[e^{i2\theta\left(\Delta_{\frac
s\ep,\frac t\ep}L(x_1,y_1) + \Delta_{0,\frac t\ep}L(\frac s\ep,y_1)
+ \Delta_{\frac s\ep,0}L(x_1,\frac t\ep)\right)}\right]
dx_1dy_1dx_2dy_2 \Bigg).
\end{align*}
Recalling that $\varphi$ is a bounded function and computing the
expectations of complex exponentials in terms of the L\'evy exponent
$\Psi(\xi)=a(\xi)+ib(\xi)$, one can easily obtain that
$|\Lambda_\ep|\leq C\, \tilde{\Lambda}_\ep$, where $C$ is a positive
constant and
\begin{align*}
\tilde{\Lambda}_\ep & := K^2 \ep^2 \int_{\frac
t\ep}^{\frac{t'}{\ep}} \int_{\frac s\ep}^{\frac{s'}{\ep}}\int_{\frac
t\ep}^{y_2} \int_{\frac s\ep}^{x_2} \sqrt{x_1x_2y_1y_2} \,
e^{-a(\theta)\left((y_2-y_1)x_1 + (x_2-x_1)y_1\right)}\\
& \qquad \qquad \times e^{-a(2\theta)\left((y_1-\frac t\ep)\frac
s\ep + (x_1-\frac s\ep)y_1\right)} dx_1dy_1dx_2dy_2.
\end{align*}
We finally prove that $\lim_{\ep\to 0}\tilde{\Lambda}_\ep=0$.
Indeed, taking into account the integration limits of all variables
and applying Fubini theorem, we have
\[
\tilde{\Lambda}_\ep \leq C\, \int_{\frac t\ep}^{\frac{t'}{\ep}}
\int_{\frac s\ep}^{\frac{s'}{\ep}}\int_{y_1}^{\frac{t'}{\ep}}
\int_{x_1}^{\frac{s'}{\ep}}
e^{-\min(a(\theta),a(2\theta))\left((y_2-\frac t\ep)\frac s\ep +
(x_2-\frac s\ep)\frac t\ep\right)} dx_2 dy_2 dx_1 dy_1.
\]
We integrate with respect to $x_2$ and $y_2$, thus
\begin{align*}
\tilde{\Lambda}_\ep & \leq C\, \ep^2 \int_{\frac
t\ep}^{\frac{t'}{\ep}} \int_{\frac s\ep}^{\frac{s'}{\ep}}
e^{-\min(a(\theta),a(2\theta))\left((y_1-\frac t\ep)\frac s\ep +
(x_1-\frac s\ep)\frac t\ep\right)} dx_1 dy_1 \\
& \leq C\, \ep^4.
\end{align*}
Hence, we have $\lim_{\ep\to 0}\tilde{\Lambda}_\ep=0$, which implies
that $\lim_{\ep\to 0}\Lambda_\ep=0$, and so $\lim_{\ep\to
0}A_\ep-B_\ep=0$. The proof is complete. \qed

\medskip

We have all needed ingredients to prove the main result of the
paper:

{\it{Proof of Theorem \ref{thm:main}}}. Taking into account the
tightness result Proposition \ref{prop:2} and Propositions
\ref{prop:5} and \ref{prop:6}, in order to apply Theorem \ref{thm:3}
we only need to prove the validity of condition \eqref{eq:8888} in
our case.

Note that $\mathsf{Re}(X)$ and $\mathsf{Im}(X)$ satisfy
\eqref{eq:8888} if, for any $(0,0)\leq (s,t)\leq (s',t')\leq (S,T)$,
$0\leq s_1<\cdots<s_n\leq s$ and $0\leq t_1<\cdots<t_n\leq t$, and
any continuous bounded function $\varphi:\bC^n\to \R$, we have
\[
\lim_{\ep\to 0} \E\left[\varphi(X_\ep(s_1,t_1),\dots,X_\ep(s_n,t_n))
\left( \Delta_{s,t} \mathsf{Re}(X_\ep)(s',t')\right) \left(
\Delta_{s,t} \mathsf{Im}(X_\ep)(s',t')\right)\right] = 0.
\]
Using the equality $\alpha \beta=\frac i4  \left\{ (\alpha-i\beta)^2
- (\alpha+i\beta)^2\right\}$, we obtain that
\begin{align}
& \E\left[\varphi(X_\ep(s_1,t_1),\dots,X_\ep(s_n,t_n)) \left(
\Delta_{s,t} \mathsf{Re}(X_\ep)(s',t')\right) \left( \Delta_{s,t}
\mathsf{Im}(X_\ep)(s',t')\right)\right]\nonumber\\
& \quad = \frac i4 \, \E
\Bigg[\varphi(X_\ep(s_1,t_1),\dots,X_\ep(s_n,t_n)) \left\{ \left(K
\ep \int_{\frac t\ep}^{\frac{t'}{\ep}} \int_{\frac
s\ep}^{\frac{s'}{\ep}} \sqrt{xy}\, e^{-i\theta L(x,y)} dx
dy\right)^2\right.\nonumber \\
& \qquad \qquad - \left.\left.\left(K \ep \int_{\frac
t\ep}^{\frac{t'}{\ep}} \int_{\frac s\ep}^{\frac{s'}{\ep}}
\sqrt{xy}\, e^{i\theta L(x,y)} dx dy\right)^2\right\}\right].
\label{eq:203}
\end{align}
We observe that, in the analysis of the term $A_\ep-B_\ep$ in the
proof of Proposition \ref{prop:6}, we indeed proved that the two
terms arising from the difference in \eqref{eq:203} converge to zero
as $\ep\to 0$. So the proof is complete. \qed


\section{Weak convergence for the stochastic heat equation}
\label{sec:heat}

We consider the following one-dimensional quasi-linear stochastic
heat equation:
\begin{equation}\label{eq:122}
\frac{\partial U}{\partial t}(t,x)-\frac{\partial ^2 U}{\partial
x^2} (t,x)=b(U(t,x))+\dot W(t,x),\quad (t,x)\in [0,T]\times [0,1],
 \end{equation}
where $T>0$ stands for a fixed time horizon, $b:\R\rightarrow \R$ is
a globally Lipschitz function and $\dot W(t,x)$ denotes the
space-time white noise. We impose the initial condition
$U(0,x)=u_0(x)$, $x\in [0,1]$, where $u_0:[0,1]\rightarrow \R$ is a
continuous function, and boundary conditions of Dirichlet type:
$$U(t,0)=U(t,1)=0,\; t\in [0,T].$$
For simplicity's sake, throughout the section we will assume that
$T=1$. All results presented here can be easily extended to a
general $T>0$.

The solution to equation \eqref{eq:122} is interpreted in the mild
sense, as follows. Let $\{W(t,x); \, (t,x)\in [0,1]^2\}$ be a
Brownian sheet defined on some probability space
$(\Omega,\mathcal{F}, \P)$  and $\{\mathcal{F}_t;\, t\in [0,1]\}$
its natural filtration. A jointly measurable and adapted random
field $U=\{U(t,x);\; (t,x)\in [0,1]^2\}$ is a solution of
\eqref{eq:122} if it holds that
\begin{align}
U(t,x) = & \int_0^1\! G_t(x,y) u_0(y) dy+\int_0^t\!\!\int_0^1\! G_{t-s}(x,y)\,b(U(s,y)) dy ds \nonumber \\
& \qquad +\int_0^t\!\!\int_0^1\!G_{t-s}(x,y)W(ds,dy), \quad a.s.
\label{eq:mild}
\end{align}
for all $(t,x)\in (0,1]\times (0,1)$, where $G$ denotes the Green
function associated to the heat equation in $[0,1]$ with Dirichlet
boundary conditions. Existence, uniqueness and pathwise continuity
of the solution to \eqref{eq:mild} are a consequence of \cite[Thm
3.5]{walsh}. For the reader's convenience, we recall that the Green
function $G$ is given by
$$G_t(x,y)=2\sum_{n=1}^\infty \sin(n\pi x) \sin(n\pi y) e^{-n^2 \pi^2
t}.$$
 In this section, we aim to apply \cite[Thm. 1.4]{BJQ} in order
to deduce that the above solution $U$ can be approximated in law, in
the space $\C([0,1]^2)$ of continuous functions, by the family of
mild solutions $\{U_n\}_{n\geq 0}$, where $U_n$ solves a stochastic
heat equation perturbed by (the formal derivative of) either the
real or imaginary part of the family introduced in \eqref{eq:1}:
\[
X_\ep(t,x)=\ep K
\int_0^{\frac{t}{\ep}}\int_0^{\frac{x}{\ep}}\sqrt{sy}\,
\{\cos(\theta L(s,y))+i\sin(\theta L(s,y))\}\,dy ds,
\]
where we recall that $\{L(s,y);\, s,y\geq 0\}$ denotes a L\'evy
sheet and its L\'evy exponent is given by
$\Psi(\xi)=a(\xi)+ib(\xi)$. The constant $K$ is given in
\eqref{eq:50} and $\theta\in (0,2\pi)$, where we assume that $a(\theta)a(2\theta)\neq 0$. Note that, compared to
\eqref{eq:1}, in the above expression of $X_\ep$ we have modified
the variables' notation in order to properly match with the
framework of stochastic partial differential equations.

Let us be more precise about the above statement. First, we rewrite
$X_\ep$ in the following way:
\[
X_\ep(t,x)=n K \int_0^t\int_0^x \sqrt{sy}\, \left\{\cos\left(\theta
L\left(\sqrt{n}\,s,\sqrt{n}\, y\right)\right)+i\sin\left(\theta
L\left(\sqrt{n}\,s,\sqrt{n}\, y\right)\right)\right\}\,dy ds,
\]
with $n=\ep^{-2}$. Set
\[
\theta_n^1(s,y):= n K \sqrt{sy}\, \cos\left(\theta
L\left(\sqrt{n}\,s,\sqrt{n}\, y\right)\right) \quad \text{and}\quad
\theta_n^2(s,y):= n K \sqrt{sy}\, \sin\left(\theta
L\left(\sqrt{n}\,s,\sqrt{n}\, y\right)\right).
\]
Let $i\in \{1,2\}$ and consider the stochastic heat equation
\begin{equation*}
\frac{\partial U^i_n}{\partial t}(t,x)-\frac{\partial ^2
U^i_n}{\partial x^2}
(t,x)=b(U^i_n(t,x))+\theta^i_n(t,x),\qquad(t,x)\in [0,1]^2,
\end{equation*}
with initial condition $u_0$ and Dirichlet boundary conditions. The
mild form of this equation is given by
\begin{align}
U^i_n(t,x)= & \int_0^1\! G_t(x,y) u_0(y) dy+\int_0^t\!\!\int_0^1\!G_{t-s}(x,y)\,b(U^i_n(s,y)) dy ds \nonumber \\
& \qquad +\int_0^t\!\!\int_0^1\!G_{t-s}(x,y)\theta^i_n(s,y) dy ds.
\label{eq:123}
\end{align}
Owing to \cite[Sec. 3]{BJQ}, equation \eqref{eq:123} admits a unique
solution $U^i_n$ whose paths are continuous almost surely. Here is
the main result of the section:

\begin{theorem}\label{thm:heat}
For any $i\in\{1,2\}$, the sequence $\{U^i_n\}_{n\geq 1}$ converges
in law, as $n\to\infty$ and in the space $\C([0,1]^2)$, to the
solution $U$ of \eqref{eq:mild}.
\end{theorem}

The proof of this theorem is based on \cite[Thm. 1.4]{BJQ}, where
sufficient conditions on a family of random fields
$\{\theta_n\}_{n\geq 1}$ have been established such that the
sequence of solutions to the stochastic heat equation driven by
$\theta_n$ converges in law to $U$, in the space of continuous
functions. The first requirement is that $\theta_n\in L^2([0,1]^2)$
a.s., and then there are the following conditions (see hypotheses
1.1, 1.2 and 1.3 in \cite{BJQ}):

\smallskip

\begin{itemize}
\item[(i)] The finite dimensional distributions of the processes
$$\zeta_n(t,x):=\int_0^t\int_0^x \theta_n(s,y) dyds,\quad (t,x)\in [0,1]^2,$$
converge in law to those of the Brownian sheet.
\item[(ii)] For some $q\in [2,3)$, there exists a positive constant $C_q$
such that, for any $f\in L^q([0,1]^2)$, it holds:
$$ \sup_{n\geq 1}\, \E\left[\left( \int_0^1 \int_0^1 f(t,x) \theta_n(t,x) \, dx dt\right)^2\right]\leq C_q \left(
\int_0^1
\int_0^1 |f(t,x)|^{q}\, dx dt\right)^{\frac 2{ q}}.$$
\item[(iii)] There exist $m> 8$ and a positive constant $C$
such that the following is satisfied: for all $s_0,\,s_0'\in [0,1]$
and $x_0,\,x_0'\in [0,1]$ satisfying $0<s_0<s_0'<2s_0$ and
$0<x_0<x_0'<2x_0$, and for any $f\in L^2([0,1]^2)$, it holds:
$$\sup_{n\geq 1} \, \E
\left[\left|\int_{s_0}^{s_0'}\int_{x_0}^{x_0'}f(s,y)\,\theta_n(s,y)dyds\right|^m\right]
\leq C \left( \int_{s_0}^{s_0'}\int_{x_0}^{x_0'}  f(s,y)^2 \, dy
ds\right)^{\frac m2}.$$
\end{itemize}

\smallskip

Hence, in the proof of Theorem \ref{thm:heat} we will prove the
validity of all above conditions in the case where $\theta_n$ is
given by $\theta_n^i$, for any $i\in \{1,2\}$. Indeed, as it will be
explained below, we will use similar arguments as those used in one
of the applications tackled in \cite{BJQ}, namely the case where
$\theta_n$ are given by the Kac-Stroock processes on the plane:
\[
\theta_n(t,x)= n \sqrt{tx} \,(-1)^{N_n(t,x)},
\]
where $N_n(t,x):= N\left(\sqrt{n}t,\sqrt{n}x\right)$, and
$\{N(t,x);\, (t,x)\in [0,1]^2\}$ is a standard Poisson process in
the plane.

We start with the following technical lemma, which is the analogous
of \cite[Lem. 4.2]{BJQ}:

\begin{lemma}
Let $f\in L^2([0,1]^2)$ and $\alpha\geq 1$. Then, for any $u,u'\in
(0,1)$ satisfying that $0<u<u'\leq 2^\alpha u$, it holds
$$\sup_{n\geq 1}\, \E\left[\left(  \int_0^1 \int_u^{u'} f(t,x) \theta^i_n(t,x)\, dx
dt\right)^2\right] \leq \frac{3}{a(\theta)^2}(2^{\alpha+1}-1) K^2
 \int_0^1 \int_u^{u'} f(t,x)^2 \,
dx dt,$$ for any $i\in \{1,2\}$.
\end{lemma}

\begin{proof}
We will only deal with the case involving $\theta_n^1$, since the
result for $\theta^2_n$ follows exactly in the same way. Note that
we clearly have
\[
\E\left[\left(  \int_0^1 \int_u^{u'} f(t,x) \theta^i_n(t,x)\, dx
dt\right)^2\right] \leq \E\left[\left| nK \int_0^1 \int_u^{u'}
\sqrt{tx}\, e^{i\theta L\left(\sqrt{n}\,t,\sqrt{n}\, x\right)}
f(t,x) \, dx dt\right|^2\right],
\]
and the latter term equals to
\begin{align*}
& n^2 K^2 \int_0^1 \int_u^{u'}\int_0^1 \int_u^{u'} \sqrt{t_1 x_1 t_2
x_2} \, \E\left[ e^{i\theta
\left(\Delta_{0,0}L\left(\sqrt{n}\,t_1,\sqrt{n}\, x_1\right) -
\Delta_{0,0}L\left(\sqrt{n}\,t_2,\sqrt{n}\, x_2\right) \right)}
\right]\\
& \qquad \qquad \qquad \times f(t_1,x_1)f(t_2,x_2) dx_1 dt_1 dx_2 dt_2.
\end{align*}
Observe that this expression is completely analogous as that at the
beginning of the first step in the proof of Lemma \ref{lem:25}.
Thus, the same arguments used therein yield
\begin{equation}\label{eq:124}
\E\left[\left(  \int_0^1 \int_u^{u'} f(t,x) \theta^i_n(t,x)\, dx
dt\right)^2\right] \leq 2\left(I_1^n + I_2^n\right),
\end{equation}
where
\begin{align*}
I_1^n & = n^2K^2 \int_0^1 \int_0^{t_2} \int_u^{u'} \int_u^{x_2}
f(t_1,x_1)f(t_2,x_2) \sqrt{t_1 x_1 t_2 x_2}\,
e^{-a(\theta)n (x_2t_2-x_1t_1)}\\
& \qquad \qquad \times \cos(b(\theta)n(x_2t_2-x_1t_1))
dx_1dx_2dt_1dt_2
\end{align*}
and
\begin{align*}
I_2^n & = n^2K^2 \int_0^1 \int_0^{t_2} \int_u^{u'} \int_u^{x_2}
f(t_1,x_1)f(t_2,x_2) \sqrt{t_1 x_1 t_2 x_2}\,
e^{-a(\theta)n \left((t_2-t_1)x_1 + (x_2-x_1)t_1\right)}\\
& \qquad \qquad \times \cos\left(b(\theta)n\left((t_2-t_1)x_1 +
(x_2-x_1)t_1\right)\right) dx_1dx_2dt_1dt_2.
\end{align*}
At this point, we apply the inequality $z w\leq \frac12 (z^2 + w^2)$
in such a way that
\[
f(t_1,x_1)f(t_2,x_2) \sqrt{t_1 x_1 t_2 x_2}\leq \frac12 \left( t_1
x_1 f(t_1,x_1)^2 + t_2 x_2 f(t_2,x_2)^2\right).
\]
This makes that both $I_1^n$ and $I_2^n$ can be bounded by the sum
of two terms of the form $I_{j,1}^n+I_{j,2}^n$, $j=1,2$,
respectively, where $I_{j,1}^n$ involves $f(t_1,x_1)$ and
$I_{j,2}^n$ involves $f(t_2,x_2)$. Then, once all cosinus are simply
bounded by 1, one observe that the resulting four terms are
completely analogous as those appearing in the proof of Lemma 4.2 in
\cite{BJQ}, and can be treated using the same kind of arguments.
Thus, we obtain that
\begin{align*}
I_{1,1}^n & \leq \frac12 \frac{1}{a(\theta)^2} K^2 \int_0^1 \int_u^{u'}
f(t_1,x_1)^2 dx_1dt_1,\\
I^n_{1,2}& \leq \frac12 \frac{2^\alpha}{a(\theta)^2} K^2 \int_0^1
\int_u^{u'} f(t_2,x_2)^2 dx_2dt_2,\\
I_{2,1}^n & \leq \frac12 \frac{2^\alpha}{a(\theta)^2} K^2 \int_0^1
\int_u^{u'} f(t_1,x_1)^2 dx_1dt_1,\\
I_{2,2}^n & \leq \frac12 \frac{4(2^\alpha-1)}{a(\theta)^2} K^2 \int_0^1
\int_u^{u'} f(t_2,x_2)^2 dx_2dt_2.
\end{align*}
Plugging everything together and using \eqref{eq:124}, we conclude
the proof.
\end{proof}

The above lemma allows us to prove the following proposition. In
fact, its proof follows exactly the same lines as that of
Proposition 4.1 in \cite{BJQ} and therefore will be omitted.
\begin{proposition}\label{prop:9}
Let $p>1$ and $f\in L^{2p}([0,1]^2)$. Then, there exists a positive
constant $C_p$ which does not depend on $f$ such that
\[
\sup_{n\geq 1}\, \E\left[\left( \int_0^1 \int_0^1 f(t,x)
\theta^i_n(t,x) \, dx dt\right)^2\right] \leq C_p \left( \int_0^1
\int_0^1 |f(t,x)|^{2p}\, dx dt\right)^{\frac 1 p},
\]
for any $i\in \{1,2\}$.
\end{proposition}

The last needed ingredient for the proof of Theorem \ref{thm:heat}
is the following result, which is the analogous of \cite[Prop.
4.4]{BJQ} in our setting.

\begin{proposition}\label{prop:10}
Let $m\in\mathbb N$ be an even number and $f\in L^2([0,1]^2)$. Then, there exists a
positive constant $C_m$ which does not depend on $f$ such that, for
all $s_0,\,s_0', x_0,\,x_0'\in [0,1]$ satisfying $0<s_0<s_0'<2s_0$
and $0<x_0<x_0'<2x_0$, we have that
$$\sup_{n\geq
1}\,
\E\left[\left(\int_{s_0}^{s_0'}\int_{x_0}^{x_0'}f(s,y)\,\theta^i_n(s,y)dyds\right)^m\right]
\leq C_m \left( \int_{s_0}^{s_0'}\int_{x_0}^{x_0'}  f(s,y)^2 \, dy
ds\right)^{\frac m2},
$$
for any $i\in \{1,2\}$.
\end{proposition}

\begin{proof}
Let $i\in \{1,2\}$. For any $(s_0,x_0)\in [0,1]^2$, we define
\[
Z^i_n(s_0,x_0):=\int_0^{s_0}\int_0^{x_0} f(s,y)\theta_n^i(s,y)dy ds.
\]
Observe that, for all $(0,0)\leq (s_0,x_0)<(s_0',x_0')\leq (1,1)$,
we have
\begin{equation}\label{eq:125}
\E\left[\left(\Delta_{s_0,x_0} Z_n^i(s_0',x_0')\right)^m\right] \leq
\E\left[\left|\Delta_{s_0,x_0} \bar{Z}_n(s_0',x_0')\right|^m\right],
\end{equation}
where the random field $\bar{Z}_n$, which does not depend on $i$, is
complex-valued and given by
\[
\bar{Z}_n(s_0,x_0):=\int_0^{s_0}\int_0^{x_0}
f(s,y)\left(\theta_n^1(s,y)+ i \theta_n^2(s,y)\right)dy ds
\]
(here $i=\sqrt{-1}$). In order to bound the right hand-side of
\eqref{eq:125}, we can proceed as in the first part of the proof of
the tightness result Proposition \ref{prop:2}, obtaining
\begin{align*}
\E\left[\left(\Delta_{s_0,x_0} Z_n^i(s_0',x_0')\right)^m\right] & \leq
n^m K^m \int_{[s_0,s_0']^m} \int_{[x_0,x_0']^m} \prod_{j=1}^m f(s_j,y_j) \sqrt{s_jy_j} \\
& \qquad \qquad \times \left| \E\left[ e^{i\theta \sum_{j=1}^m (-1)^j \Delta_{0,0}L(\sqrt{n}\, s_j,\sqrt{n}\, y_j)}
\right]\right| dy_1 \cdots dy_m ds_1\cdots ds_m.
\end{align*}
At this point, we apply that $y_j<x_0'<2x_0$ and $s_j<s_0'<2s_0$, and we compute the modulus of the
expectation as it has been done in the proof of Proposition \ref{prop:2}; more precisely, using the method set up
therein in order to end up with the estimate \eqref{eq:3}. Thus, we can infer that
\begin{align*}
&\E\left[\left(\Delta_{s_0,x_0} Z_n^i(s_0',x_0')\right)^m\right] \\
& \qquad \leq
2^m (s_0x_0)^{\frac m2} n^m K^m
\int_{[s_0,s_0']^m} \int_{[x_0,x_0']^m} \prod_{j=1}^m f(s_j,y_j) \, e^{-a(\theta) n s_0 \left((y_{(m)}-y_{(m-1)})+\cdots+(y_{(2)}-y_{(1)})\right)}\\
& \qquad \qquad\qquad \qquad \qquad \times e^{-a(\theta) n x_0 \left((s_{(m)}-s_{(m-1)})+\cdots+(s_{(2)}-s_{(1)})\right)}\,
dy_1 \cdots dy_m ds_1\cdots ds_m\\
& \qquad =
2^m (s_0x_0)^{\frac m2} m!\, n^m K^m
\int_{[s_0,s_0']^m} \int_{[x_0,x_0']^m} \prod_{j=1}^m f(s_j,y_j)  \,
e^{-a(\theta) n s_0 \left((y_{(m)}-y_{(m-1)})+\cdots+(y_{(2)}-y_{(1)})\right)}\\
& \qquad \qquad \qquad \qquad \qquad\times e^{-a(\theta) n x_0 \left((s_m-s_{m-1})+\cdots+(s_{2}-s_{1})\right)}\,
1_{\{s_1\leq \cdots\leq s_m\}} \, dy_1 \cdots dy_m ds_1\cdots ds_m.
\end{align*}
Note that the latter expression is almost equal to that in the right
hand-side of equation (31) in the proof of \cite[Prop. 4.4]{BJQ}.
Hence, we can conclude the proof exactly in the same way as in that
result.
\end{proof}

\smallskip

{\it{Proof of Theorem \ref{thm:heat}.}} As explained above, we need that $\theta^i_n\in L^2([0,1]^2)$, a.s.,
which is clear by definition of the random fields $\theta^i_n$, $i=1,2$, and that conditions (i), (ii) and
(iii) are fulfilled.

First, note that (i) is a consequence of Theorem \ref{thm:main}.
Secondly, Proposition \ref{prop:9} implies that (ii) is satisfied
and, finally, Proposition \ref{prop:10} assures the validity of
condition (iii). \qed


\section*{Acknowledgement}

The authors thank the anonymous referee for a careful reading of the
manuscript.


\end{document}